\documentclass[11pt]{amsart}
\usepackage[T1]{fontenc}
\usepackage[utf8]{inputenc}
\usepackage[
  backend=biber,
  style=numeric,
  sorting=none,
  giveninits=false,
  doi=true,
  isbn=true,
  url=true,
  eprint=true
]{biblatex}
\usepackage{lmodern}
\usepackage{microtype}
\usepackage{tabularx}
\usepackage{array}
\usepackage{amsmath,amssymb,amsfonts,amsthm}
\usepackage{mathtools}
\usepackage{enumitem}
\usepackage{tikz-cd}
\usepackage{geometry}
\usepackage{float}
\usepackage[colorlinks=true,linkcolor=blue,citecolor=blue,urlcolor=blue]{hyperref}
\usepackage[nameinlink,noabbrev]{cleveref}

\hypersetup{
  pdftitle={From Finite-Node Conifold Geometry to BPS Structures II: Functorial Incidence and Quiver Assembly},
  pdfauthor={Abdul Rahman}
}
\numberwithin{table}{section}

\numberwithin{equation}{section}

\theoremstyle{plain}
\newtheorem{theorem}{Theorem}[section]
\newtheorem{proposition}[theorem]{Proposition}
\newtheorem{lemma}[theorem]{Lemma}
\newtheorem{corollary}[theorem]{Corollary}

\theoremstyle{definition}
\newtheorem{definition}[theorem]{Definition}

\theoremstyle{remark}
\newtheorem{remark}{Remark}[section]

\makeatletter



\makeatother

\newcommand{\Q}{\mathbb{Q}}

\newcommand{\Perv}{\mathrm{Perv}}

\newcommand{\Ext}{\mathrm{Ext}}

\newcommand{\rat}{\mathrm{rat}}

\title[From Finite-Node Conifold Geometry to BPS Structures II]{From Finite-Node Conifold Geometry to BPS Structures II: Functorial Incidence and Quiver Assembly}
\author{Abdul Rahman}
\thanks{Email: arahman@alum.howard.edu}
\subjclass[2020]{14D06, 32S30, 14F43, 18G80}
\keywords{Conifold degeneration, perverse sheaves, mixed Hodge modules, perverse schobers, incidence, quiver assembly}

\begin{document}

\begin{abstract}
In previous work, we extracted the intrinsic finite algebraic state data of a finite-node conifold degeneration in the form $A_\Sigma := (V_\Sigma,E_\Sigma,c_\Sigma)$, where $V_\Sigma$ is the finite node-indexed vertex set, $E_\Sigma$ is the nodewise coupling space, and $c_\Sigma$ is the coefficient vector of the corrected global extension class. The purpose of the present paper is to construct the corresponding interaction and incidence layer. Starting from the finite-node schober package $S_\Sigma := (\mathcal C_{\mathrm{bulk}},\{\mathcal C_{p_k}\}_{k=1}^r,\{\Phi_k,\Psi_k\}_{k=1}^r,Sh(S_\Sigma))$, we define the extended vertex set $V_\Sigma^{\mathrm{ext}} := V_\Sigma \sqcup \{v_{\mathrm{bulk}}\}$, the functorial coupling relation determined by the attachment functors, the resulting functorial incidence package $\mathfrak{I}_\Sigma := (V_\Sigma^{\mathrm{ext}},\rightsquigarrow_\Sigma)$, and its canonical binary decategorification $\mathcal I_\Sigma := (V_\Sigma^{\mathrm{ext}},I_\Sigma)$. From these data we assemble the finite quiver-theoretic package $\mathfrak Q_\Sigma := (V_\Sigma,E_\Sigma,c_\Sigma,\mathcal F_\Sigma,I_\Sigma)$, where $\mathcal F_\Sigma := \{(\Phi_k,\Psi_k)\}_{k=1}^r$ is the functorial coupling datum. We prove that this package is canonically determined by the finite-node schober datum, compatible with the corrected perverse extension and its mixed-Hodge-module refinement, and invariant under equivalence of finite-node schober realizations. This yields the interaction and incidence layer required for later graded interaction, stability, BPS, and wall-crossing structures.
\end{abstract}
\maketitle
\tableofcontents
\bigskip

\section{Introduction}

Let $\pi:X\to\Delta$ be a finite-node conifold degeneration whose central fiber $X_0$ has ordinary double points (ODP) $\Sigma=\{p_1,\dots,p_r\}\subset X_0$. In earlier work \cite{RahmanQuiverDataI}, we extracted from this degeneration the intrinsic finite algebraic state-data package $A_\Sigma:=(V_\Sigma,E_\Sigma,c_\Sigma)$, where $V_\Sigma=\{v_1,\dots,v_r\}$ is the finite vertex set indexed by the node set $\Sigma$, $E_\Sigma=\Ext^1_{\Perv(X_0;\Q)}(Q_\Sigma,IC_{X_0})$ is the nodewise coupling space, and $c_\Sigma=(c_1,\dots,c_r)\in\Q^r$ is the coefficient vector of the corrected global extension class. More precisely, the corrected finite-node perverse extension
\[
0\to IC_{X_0}\to \mathcal P\to Q_\Sigma\to 0,
\qquad
Q_\Sigma:=\bigoplus_{k=1}^r i_{k*}\Q_{\{p_k\}},
\]
together with the decomposition $E_\Sigma\cong\bigoplus_{k=1}^r \Q e_k$ and the expansion $[\mathcal P]_{\mathrm{perv}}=\sum_{k=1}^r c_k e_k$, determines the package $A_\Sigma$ canonically; see \cite{RahmanQuiverDataI,RahmanPerverseNearbyCycles,RahmanMixedHodgeModules}.

The purpose of the present paper is to construct the corresponding interaction and incidence layer. The key new input is the finite-node schober realization. By \cite{RahmanMultiNodeSchoberPaper}, the same finite-node degeneration determines a schober package $S_\Sigma:=\bigl(\mathcal C_{\mathrm{bulk}},\{\mathcal C_{p_k}\}_{k=1}^r,\{\Phi_k,\Psi_k\}_{k=1}^r,Sh(S_\Sigma)\bigr)$ with one localized categorical sector $\mathcal C_{p_k}$ for each node $p_k\in\Sigma$, a common bulk category $\mathcal C_{\mathrm{bulk}}$, attachment functors $\Phi_k:\mathcal C_{p_k}\to\mathcal C_{\mathrm{bulk}}$ and $\Psi_k:\mathcal C_{\mathrm{bulk}}\to\mathcal C_{p_k}$, and shadow $Sh(S_\Sigma)\cong \mathcal P$. In \cite{RahmanQuiverDataI} this categorical datum entered only as a compatible realization of the same finite-node architecture. The aim of the present paper is to make the functorial content of that realization mathematically active.

\subsection{From state data to interaction data}

The state-data package $A_\Sigma=(V_\Sigma,E_\Sigma,c_\Sigma)$ records the finite algebraic variables canonically attached to the degeneration, but it does not yet record how the localized sectors interact. The attachment functors $\Phi_k$ and $\Psi_k$ provide precisely the additional information needed to address that problem. The first step is therefore to enlarge the vertex set by adjoining a bulk vertex $v_{\mathrm{bulk}}$ and defining the extended vertex set $V_\Sigma^{\mathrm{ext}}:=V_\Sigma\sqcup\{v_{\mathrm{bulk}}\}$. One then uses the attachment functors to define the corresponding functorial coupling relation. At the most basic level, the datum of $\Phi_k$ and $\Psi_k$ determines the bulk couplings $v_k\rightsquigarrow v_{\mathrm{bulk}}$ and $v_{\mathrm{bulk}}\rightsquigarrow v_k$, and once compositions are admitted, the composites $\Psi_j\circ\Phi_i$ determine mediated node-to-node couplings $v_i\rightsquigarrow v_j$.

The mathematical problem addressed in this paper is to show that these relations are canonical consequences of the finite-node schober datum and that they admit an algebraic decategorification compatible with the state-data package extracted in \cite{RahmanQuiverDataI}. In this way the present paper produces the finite interaction layer carried by the degeneration. If \cite{RahmanQuiverDataI} isolated the intrinsic state variables, then the present paper isolates the intrinsic incidence data. These data are the missing algebraic layer between the corrected perverse/Hodge/schober architecture and the later stability, BPS-spectrum, and wall-crossing formalisms.

\subsection{Relation to earlier work}

The starting point remains the corrected perverse object first constructed in the ODP case in \cite{RahmanSchoberPaper} and then placed into the nearby-/vanishing-cycle and limiting mixed-Hodge-theoretic framework in \cite{RahmanPerverseNearbyCycles}. The finite-node extension picture and its mixed-Hodge-module refinement were established in \cite{RahmanMixedHodgeModules}, where the corrected global extension was shown to fit into
\[
0\to IC_{X_0}\to \mathcal P\to \bigoplus_{k=1}^r i_{k*}\Q_{\{p_k\}}\to 0
\]
and the nodewise decomposition of the global extension space was made explicit. In \cite{RahmanQuiverDataI} we extracted from this theorem package the finite algebraic state-data package $A_\Sigma=(V_\Sigma,E_\Sigma,c_\Sigma)$. The finite-node schober package was constructed in \cite{RahmanMultiNodeSchoberPaper}, where one proves the existence of one localized categorical sector per node, the shadow identification $Sh(S_\Sigma)\cong \mathcal P$, and the first quiver-shadow precursor.

The present paper is the next layer above that theorem package. It does not reprove the corrected extension, the mixed-Hodge-module lift, the algebraic state-data extraction of \cite{RahmanQuiverDataI}, or the existence of the finite-node schober package. Instead, it uses those results as input and extracts from them the canonical interaction/incidence structure carried by the functorial bulk/localized couplings. In particular, the paper is designed to sit between \cite{RahmanQuiverDataI}, where the finite algebraic state variables were isolated, and the later stability/BPS/wall-crossing paper, where those variables and incidence data will be used dynamically.

\subsection{Primitive local anchoring and functorial coupling}

The interaction and incidence data extracted in the present paper should not be read as externally imposed graph-theoretic decorations on the state-data package of \cite{RahmanQuiverDataI}. Rather, they are the first functorial organization of the same finite-node architecture already visible on the perverse, mixed-Hodge, and categorical sides. In \cite{RahmanQuiverDataI}, the nodewise basis
\[
E_\Sigma \cong \bigoplus_{k=1}^r \Q e_k
\]
was extracted as the first algebraic shadow of the localized finite-node package visible on the perverse, mixed-Hodge, and categorical sides. In the present paper, the new interaction layer is likewise not introduced from outside the geometry: it is carried by the localized categorical sectors
\[
\{\mathcal C_{p_k}\}_{k=1}^r,
\]
the common bulk category $\mathcal C_{\mathrm{bulk}}$, and the attachment functors
\[
\Phi_k : \mathcal C_{p_k} \to \mathcal C_{\mathrm{bulk}},
\qquad
\Psi_k : \mathcal C_{\mathrm{bulk}} \to \mathcal C_{p_k}.
\]

Thus the passage from the state-data package
\[
A_\Sigma := (V_\Sigma, E_\Sigma, c_\Sigma)
\]
to the interaction/incidence package of the present paper should be understood as the first functorial organization of the primitive local package already carried by the finite-node degeneration. The later stability, BPS, and wall-crossing papers will therefore not begin from an abstract quiver chosen independently of the geometry. They will consume an interaction layer extracted from the same corrected finite-node perverse extension, its mixed-Hodge-module refinement, and its schober realization.

\subsection{Role of current work in the larger sequence}

The present work is the first interaction-theoretic paper in the finite-node algebraic sequence. \cite{RahmanQuiverDataI} isolated the intrinsic state variables attached to the degeneration, namely the finite vertex set, the nodewise coupling space, and the coefficient vector of the corrected global extension class. The role of the present paper is to add the next layer forced by the finite-node schober package: the functorial coupling pattern, its incidence relation, its canonical binary decategorification, and the resulting quiver-theoretic package
\[
\mathfrak Q_\Sigma := (V_\Sigma, E_\Sigma, c_\Sigma, \mathcal F_\Sigma, I_\Sigma).
\]

In this sense, the present work sits exactly between the state-data paper and the later graded-interaction and dynamical papers. It does not yet define stability data, BPS sectors, or wall-crossing structure. Instead, it isolates the interaction layer without which those later constructions cannot be formulated intrinsically in finite-node terms. Thus the sequence
\[
\text{state data} \longrightarrow \text{interaction/incidence data} \longrightarrow \text{later dynamical structure}
\]
should be read as a forced progression rather than as a sequence of externally imposed algebraic enrichments.

\subsection{Relation to the quiver and BPS literature}

In the BPS-quiver literature, quivers and related incidence data provide the finite algebraic interface through which one studies constituent sectors, interaction channels, stability, and wall crossing; see, for example, \cite{AlimCecottiCordovaEspahbodiRastogiVafa_BPSQuivers,AlimCecottiCordovaEspahbodiRastogiVafa_N2Quivers,Denef_QuantumQuivers,Cecotti_QuiverBPS}. The present paper remains one step before the full BPS formalism. Its purpose is not yet to define stability conditions or BPS indices, but rather to extract the interaction layer that such a theory would consume.

The crucial point is that the finite-node schober datum already carries more than isolated categorical sectors: it carries attachment functors and hence a functorial coupling pattern. The present paper formalizes that pattern in a way that is compatible with the state-data package of \cite{RahmanQuiverDataI}. The resulting incidence data are therefore not postulated externally; they are extracted from the finite-node geometry through the corrected perverse extension, its mixed-Hodge-module refinement, and its schober realization.

\subsection{What the current work does not yet do}

The present paper deliberately stops at the interaction/incidence layer. It does not yet define a graded interaction law, does not yet impose a numerical arrow multiplicity law stronger than binary support, and does not yet introduce stability conditions, chamber structures, BPS indices, or wall-crossing formulas. Those later structures require additional categorical or homological input beyond the support-level interaction data extracted here.

More precisely, the present paper shows that the finite-node schober package already determines a canonical functorial coupling relation and hence a canonical incidence package. This is the correct first interaction-theoretic output of the finite-node geometry. The passage from support-level incidence to stronger graded or weighted interaction packages belongs to the next stage of the sequence, and the passage from those packages to stability, BPS, and wall-crossing structure belongs to the later dynamical papers.

\subsection{Why binary incidence is the correct output at this stage}

The binary decategorification adopted in the present paper is deliberate. From the finite-node schober package one canonically obtains a finite relation on the extended vertex set, and the mathematically minimal decategorification of such a relation is its characteristic function. This yields a \(\{0,1\}\)-valued incidence matrix recording only the presence or absence of coupling channels. At the level of the current theorem package, this is the only numerical incidence datum canonically forced by the constructions.

Accordingly, the present paper does not impose a stronger weighted interaction law by hand. Such a law would require additional invariants --- for example, ranks, dimensions, Euler characteristics, or graded categorical data attached to the mediated channels --- that are not yet part of the present finite-node theorem package. The binary incidence package constructed here should therefore be read as the correct support-level interaction object from which later weighted or graded refinements may be extracted.

\subsection{Main results}

The main results of the paper are the following.

\begin{enumerate}
\item[\textbf{(R1)}] The finite-node schober package determines an extended vertex set $V_\Sigma^{\mathrm{ext}}:=V_\Sigma\sqcup\{v_{\mathrm{bulk}}\}$, where the added bulk vertex corresponds to the bulk category $\mathcal C_{\mathrm{bulk}}$.

\item[\textbf{(R2)}] The attachment functors $\Phi_k:\mathcal C_{p_k}\to \mathcal C_{\mathrm{bulk}}$ and $\Psi_k:\mathcal C_{\mathrm{bulk}}\to \mathcal C_{p_k}$ determine a canonical functorial coupling relation on $V_\Sigma^{\mathrm{ext}}$, and the composites $\Psi_j\circ\Phi_i$ determine the corresponding mediated node-to-node relations.

\item[\textbf{(R3)}] This functorial coupling relation determines a canonical functorial incidence package $\mathfrak I_\Sigma:=\bigl(V_\Sigma^{\mathrm{ext}},\rightsquigarrow_\Sigma\bigr)$, which admits a canonical binary decategorification in the form of the incidence matrix $I_\Sigma\in M_{r+1}(\{0,1\})$, equivalently the decategorified incidence package $\mathcal I_\Sigma:=\bigl(V_\Sigma^{\mathrm{ext}},I_\Sigma\bigr)$.

\item[\textbf{(R4)}] The state-data package of \cite{RahmanQuiverDataI} together with the incidence data of the present paper determine a quiver-theoretic package $\mathfrak Q_\Sigma:=(V_\Sigma,E_\Sigma,c_\Sigma,\mathcal F_\Sigma,I_\Sigma)$, where $\mathcal F_\Sigma:=\{(\Phi_k,\Psi_k)\}_{k=1}^r$.

\item[\textbf{(R5)}] The incidence and quiver-theoretic package is compatible with the corrected perverse extension, its mixed-Hodge-module lift, and its schober shadow, and is invariant under the appropriate notion of equivalence of finite-node schober realizations.
\end{enumerate}

\subsection{Scope and organization}

The present paper is confined to the extraction of functorial coupling and incidence data from the finite-node schober package and their decategorified algebraic image. It does not yet define a stability structure, BPS indices, or a wall-crossing theorem. Those belong to the next paper in the sequence. The role of the present paper is to isolate the interaction layer that lies between the intrinsic state variables of \cite{RahmanQuiverDataI} and the stability/BPS formalism in future work.

Section~\ref{sec:finite-node-input} recalls the finite-node input data from \cite{RahmanQuiverDataI} and from the schober theorem package. Section~\ref{sec:extended-vertex-data} defines the bulk vertex and the extended vertex set. Section~\ref{sec:functorial-coupling-relation} introduces the functorial coupling relation determined by the attachment functors. Section~\ref{sec:functorial-incidence-package} packages these relations into the canonical incidence datum. Section~\ref{sec:decategorified-incidence-data} defines the corresponding algebraic incidence data. Section~\ref{sec:quiver-theoretic-package} assembles the quiver-theoretic package. Section~\ref{sec:compatibility-with-partI-data} proves compatibility with the state-data package of \cite{RahmanQuiverDataI}. Section~\ref{sec:invariance-under-equivalence} proves invariance under the appropriate equivalence notion. Section~\ref{sec:examples} gives explicit one-node, two-node, and three-node examples. Appendix~\ref{app:binary-decategorification} explains the mathematical status of the binary decategorification adopted in the paper and formulates the conditions required for a later weighted refinement.

\section{Finite-node input data}\label{sec:finite-node-input}

In this section we record the finite-node input data used throughout the paper. The purpose is not to reprove the results of \cite{RahmanQuiverDataI} or of the preceding papers, but to fix the notation and structural packages that serve as input to the interaction and incidence constructions developed below. The point of departure is the finite algebraic state-data package extracted in \cite{RahmanQuiverDataI}, together with the finite-node schober package constructed in \cite{RahmanMultiNodeSchoberPaper}.

\subsection{Finite node set and state-data package}

Let $\pi:X\to\Delta$ be a finite-node conifold degeneration in the sense fixed in \cite{RahmanQuiverDataI}, and let $\Sigma=\{p_1,\dots,p_r\}\subset X_0$ be the finite singular locus of the central fiber. Earlier work \cite{RahmanQuiverDataI} attached to this degeneration the finite algebraic state-data package
\[
A_\Sigma:=(V_\Sigma,E_\Sigma,c_\Sigma),
\]
whose components are recalled as follows.

First, the finite vertex set is
\begin{equation}\label{eq:partII-def-VSigma}
V_\Sigma:=\{v_1,\dots,v_r\},
\end{equation}
with $v_k$ corresponding to the node $p_k\in\Sigma$.

Second, the nodewise coupling space is
\begin{equation}\label{eq:partII-def-ESigma}
E_\Sigma:=\Ext^1_{\Perv(X_0;\Q)}(Q_\Sigma,IC_{X_0}),
\qquad
Q_\Sigma:=\bigoplus_{k=1}^r i_{k*}\Q_{\{p_k\}},
\end{equation}
and \cite{RahmanQuiverDataI} proved that
\begin{equation}\label{eq:partII-ESigma-decomposition}
E_\Sigma\cong \bigoplus_{k=1}^r \Q e_k.
\end{equation}
Thus $E_\Sigma$ is a finite-dimensional $\Q$-vector space of dimension $r$, canonically decomposed into one distinguished nodewise coupling line $\Q e_k$ for each node $p_k\in\Sigma$; see \cite{RahmanQuiverDataI,RahmanMixedHodgeModules}.

Third, the corrected global perverse extension class satisfies
\begin{equation}\label{eq:partII-cSigma-expansion}
[\mathcal P]_{\mathrm{perv}}=\sum_{k=1}^r c_k e_k,
\end{equation}
and the coefficient vector is
\begin{equation}\label{eq:partII-def-cSigma}
c_\Sigma:=(c_1,\dots,c_r)\in\Q^r.
\end{equation}

We therefore regard
\[
A_\Sigma=(V_\Sigma,E_\Sigma,c_\Sigma)
\]
as the finite algebraic state-data package attached to the degeneration. This is the algebraic input produced in \cite{RahmanQuiverDataI} and used as the state layer throughout the present paper.

\subsection{Finite-node schober package}

The categorical input is the finite-node schober package constructed in \cite{RahmanMultiNodeSchoberPaper}. We recall only the notation needed below.

\begin{definition}[finite-node schober package]
\label{def:partII-finite-node-schober-package}
The \emph{finite-node schober package} attached to the degeneration is the datum
\begin{equation}\label{eq:partII-def-SSigma}
S_\Sigma:=\Bigl(\mathcal C_{\mathrm{bulk}},\{\mathcal C_{p_k}\}_{k=1}^r,\{\Phi_k,\Psi_k\}_{k=1}^r,Sh(S_\Sigma)\Bigr),
\end{equation}
where:
\begin{enumerate}
\item $\mathcal C_{\mathrm{bulk}}$ is the bulk category attached to the smooth sector of the degeneration;
\item for each $1\le k\le r$, $\mathcal C_{p_k}$ is the localized categorical sector attached to the node $p_k\in\Sigma$;
\item $\Phi_k:\mathcal C_{p_k}\to \mathcal C_{\mathrm{bulk}}$ and $\Psi_k:\mathcal C_{\mathrm{bulk}}\to \mathcal C_{p_k}$ are the attachment functors;
\item $Sh(S_\Sigma)$ denotes the perverse-sheaf shadow of the schober datum.
\end{enumerate}
\end{definition}

By the main shadow theorem of \cite{RahmanMultiNodeSchoberPaper}, one has
\begin{equation}\label{eq:partII-shadow-is-P}
Sh(S_\Sigma)\cong \mathcal P,
\end{equation}
where $\mathcal P$ is the corrected finite-node perverse object. Thus the schober package is not independent of the corrected extension picture; it is a categorical realization of the same finite-node structure.

\begin{remark}
\label{rem:partII-same-node-indexing}
The same finite node set $\Sigma=\{p_1,\dots,p_r\}$ indexes all three levels of structure used in the present paper:
\begin{enumerate}
\item the localized quotient
\[
Q_\Sigma=\bigoplus_{k=1}^r i_{k*}\Q_{\{p_k\}}
\]
on the perverse side;
\item the localized Hodge quotient
\[
Q_\Sigma^H=\bigoplus_{k=1}^r i_{k*}\Q^H_{\{p_k\}}(-1)
\]
on the mixed-Hodge side;
\item the localized categorical family
\[
\mathcal C_\Sigma:=\{\mathcal C_{p_k}\}_{k=1}^r
\]
on the schober side.
\end{enumerate}
This common indexing, established in \cite{RahmanQuiverDataI}, is the basis of the interaction/incidence extraction carried out in the later sections.
\end{remark}

\subsection{Functorial coupling datum}

The new input for the present paper is the functorial coupling datum carried by the attachment functors of the schober package.

\begin{definition}[functorial coupling datum] \label{def:partII-functorial-coupling-datum}
Associated with the finite-node schober package \(S_\Sigma\) is the family
\[
\mathcal{F}_\Sigma := \{(\Phi_k,\Psi_k)\}_{k=1}^r,
\]
where
\[
\Phi_k : \mathcal C_{p_k} \to \mathcal C_{\mathrm{bulk}},
\qquad
\Psi_k : \mathcal C_{\mathrm{bulk}} \to \mathcal C_{p_k}.
\]
We call \(\mathcal{F}_\Sigma\) the functorial coupling datum of the finite-node degeneration.
\end{definition}

\begin{remark}
In earlier work, the data \(\mathcal{F}_\Sigma\) appeared only as part of the categorical realization of the corrected extension package. In the present paper, \(\mathcal{F}_\Sigma\) becomes mathematically active: it is the source of the functorial coupling relation, the functorial incidence package, the decategorified incidence package, and the quiver-theoretic assembly. Thus the passage from the state-data package
\[
A_\Sigma=(V_\Sigma,E_\Sigma,c_\Sigma)
\]
to the interaction layer of the present paper is exactly the passage to the structure determined by
\[
\mathcal{F}_\Sigma=\{(\Phi_k,\Psi_k)\}_{k=1}^r.
\]
\end{remark}

\subsection{Input packages for the present paper}

We summarize the input packages used below.

\begin{definition}[current input packages]
The present paper takes as input the following data:
\begin{enumerate}
\item the finite node set
\[
\Sigma := \{p_1,\dots,p_r\};
\]
\item the finite algebraic state-data package
\[
A_\Sigma := (V_\Sigma,E_\Sigma,c_\Sigma);
\]
\item the finite-node schober package
\[
S_\Sigma := \Bigl(\mathcal C_{\mathrm{bulk}},\{\mathcal C_{p_k}\}_{k=1}^r,\{\Phi_k,\Psi_k\}_{k=1}^r,Sh(S_\Sigma)\Bigr);
\]
\item the functorial coupling datum
\[
\mathcal{F}_\Sigma := \{(\Phi_k,\Psi_k)\}_{k=1}^r.
\]
\end{enumerate}
\end{definition}

\begin{remark}
The results of earlier work are used here only through the input packages of the preceding definition. In particular, the present paper does not reprove the corrected finite-node extension, the mixed-Hodge-module lift, or the existence of the finite-node schober package. Its purpose is to extract from these already-established structures the canonical interaction and incidence data they determine.
\end{remark}
\section{Extended vertex data}\label{sec:extended-vertex-data}

The purpose of this section is to enlarge the finite vertex set $V_\Sigma=\{v_1,\dots,v_r\}$ of the state-data package by adjoining a single bulk vertex corresponding to the bulk category $\mathcal C_{\mathrm{bulk}}$ of the finite-node schober package. This enlargement is necessary for three reasons. First, the attachment functors $\Phi_k:\mathcal C_{p_k}\to \mathcal C_{\mathrm{bulk}}$ and $\Psi_k:\mathcal C_{\mathrm{bulk}}\to \mathcal C_{p_k}$ are functors between localized node sectors and a common bulk sector, so the bulk sector must itself be represented at the level of vertex data. Second, the bulk category acts as the common mediator through which localized sectors are coupled. Third, the resulting incidence structure is naturally framed by the bulk sector, and this framing must be visible already at the level of vertices. Thus the correct vertex package for the present paper is not only the node-indexed set $V_\Sigma$, but its one-point bulk extension.

\subsection{The bulk vertex}

We begin by singling out the bulk categorical sector in vertex form.

\begin{definition}[bulk vertex]
\label{def:bulk-vertex}
Let
\[
S_\Sigma:=\Bigl(\mathcal C_{\mathrm{bulk}},\{\mathcal C_{p_k}\}_{k=1}^r,\{\Phi_k,\Psi_k\}_{k=1}^r,Sh(S_\Sigma)\Bigr)
\]
be the finite-node schober package of Definition~\ref{def:partII-finite-node-schober-package}. The \emph{bulk vertex} associated with $S_\Sigma$ is the formal symbol
\begin{equation}\label{eq:def-vbulk}
v_{\mathrm{bulk}}.
\end{equation}
\end{definition}

\begin{remark}
\label{rem:vbulk-formal-symbol}
The bulk vertex $v_{\mathrm{bulk}}$ is not introduced as a new geometric object on the central fiber. It is the algebraic/categorical vertex representative of the bulk category $\mathcal C_{\mathrm{bulk}}$. In particular, it plays for the bulk sector the same role that the vertices $v_k\in V_\Sigma$ play for the localized sectors indexed by the nodes $p_k\in\Sigma$.
\end{remark}

\begin{remark}
\label{rem:vbulk-why-needed}
The introduction of $v_{\mathrm{bulk}}$ is forced by the form of the attachment functors. Since each $\Phi_k$ has source $\mathcal C_{p_k}$ and target $\mathcal C_{\mathrm{bulk}}$, and each $\Psi_k$ has source $\mathcal C_{\mathrm{bulk}}$ and target $\mathcal C_{p_k}$, the coupling pattern carried by the schober datum is not purely node-to-node. It is node-to-bulk and bulk-to-node. Therefore any vertex-level incidence formalism must contain a bulk vertex.
\end{remark}

\subsection{The extended vertex set}

We now enlarge the finite vertex set of \cite{RahmanQuiverDataI} by adjoining the bulk vertex.

\begin{definition}[extended vertex set]
\label{def:extended-vertex-set}
Let $V_\Sigma=\{v_1,\dots,v_r\}$ be the finite vertex set of the state-data package $A_\Sigma=(V_\Sigma,E_\Sigma,c_\Sigma)$ from \cite{RahmanQuiverDataI}. The \emph{extended vertex set} is defined by
\begin{equation}\label{eq:def-extended-vertex-set}
V_\Sigma^{\mathrm{ext}}:=V_\Sigma\sqcup\{v_{\mathrm{bulk}}\}.
\end{equation}
Equivalently,
\[
V_\Sigma^{\mathrm{ext}}:=\{v_1,\dots,v_r,v_{\mathrm{bulk}}\}.
\]
\end{definition}

\begin{lemma}[finiteness of the extended vertex set]
\label{lem:finiteness-extended-vertex-set}
If $|\Sigma|=r<\infty$, then
\[
|V_\Sigma^{\mathrm{ext}}|=r+1.
\]
\end{lemma}

\begin{proof}
By definition, $V_\Sigma=\{v_1,\dots,v_r\}$ has cardinality $r$. The extended vertex set is obtained by adjoining the single bulk vertex $v_{\mathrm{bulk}}$, which is distinct from the node-indexed vertices. Hence $|V_\Sigma^{\mathrm{ext}}|=|V_\Sigma|+1=r+1$.
\end{proof}

\begin{proposition}[canonical decomposition of the extended vertex set]
\label{prop:extended-vertex-decomposition}
The extended vertex set admits the canonical decomposition
\begin{equation}\label{eq:extended-vertex-decomposition}
V_\Sigma^{\mathrm{ext}}:=V_\Sigma\sqcup\{v_{\mathrm{bulk}}\},
\end{equation}
where $V_\Sigma$ is the localized vertex set indexed by the node set $\Sigma=\{p_1,\dots,p_r\}$ and $v_{\mathrm{bulk}}$ is the bulk vertex attached to $\mathcal C_{\mathrm{bulk}}$.
\end{proposition}

\begin{proof}
This is immediate from Definitions~\ref{def:bulk-vertex} and \ref{def:extended-vertex-set}. The node-indexed vertices and the bulk vertex serve different roles and are therefore kept disjoint by construction.
\end{proof}

\begin{remark}
\label{rem:localized-bulk-split-vertex-level}
The decomposition $V_\Sigma^{\mathrm{ext}}:=V_\Sigma\sqcup\{v_{\mathrm{bulk}}\}$ is the vertex-level algebraic form of the bulk/localized distinction already visible geometrically in the decomposition $X_0=U\sqcup\Sigma$; see \cite{RahmanQuiverDataI}. It is likewise visible sheaf-theoretically in the corrected extension
\[
0\to IC_{X_0}\to \mathcal P\to Q_\Sigma\to 0,
\]
and categorically in the schober package $S_\Sigma:=\Bigl(\mathcal C_{\mathrm{bulk}},\{\mathcal C_{p_k}\}_{k=1}^r,\{\Phi_k,\Psi_k\}_{k=1}^r,Sh(S_\Sigma)\Bigr)$. At the level of vertices, $V_\Sigma$ records the localized node sectors and $v_{\mathrm{bulk}}$ records the common bulk sector.
\end{remark}

\subsection{Why the bulk vertex is required}

We now explain why the extended vertex set rather than the localized vertex set alone is the correct input for the interaction formalism of the current work.

\begin{proposition}[necessity of the bulk vertex]
\label{prop:necessity-of-bulk-vertex}
Let $\mathcal F_\Sigma:=\{(\Phi_k,\Psi_k)\}_{k=1}^r$ be the functorial coupling datum of Definition~\ref{def:partII-functorial-coupling-datum}. Then the datum $\mathcal F_\Sigma$ cannot be represented purely on the localized vertex set $V_\Sigma$ alone. Rather, its source-target structure requires the extended vertex set $V_\Sigma^{\mathrm{ext}}$.
\end{proposition}

\begin{proof}
For each $1\le k\le r$, the attachment functors satisfy $\Phi_k:\mathcal C_{p_k}\to \mathcal C_{\mathrm{bulk}}$ and $\Psi_k:\mathcal C_{\mathrm{bulk}}\to \mathcal C_{p_k}$. Thus the coupling datum has one endpoint in the localized sector $\mathcal C_{p_k}$ and the other endpoint in the bulk category $\mathcal C_{\mathrm{bulk}}$. The localized vertex set $V_\Sigma=\{v_1,\dots,v_r\}$ contains only the node-indexed representatives of the categories $\mathcal C_{p_k}$ and does not contain a representative of $\mathcal C_{\mathrm{bulk}}$. Therefore it is insufficient for recording the source-target pattern of the attachment functors. Adjoining $v_{\mathrm{bulk}}$ remedies exactly this deficiency and yields the extended vertex set $V_\Sigma^{\mathrm{ext}}:=V_\Sigma\sqcup\{v_{\mathrm{bulk}}\}$.
\end{proof}

\begin{remark}
\label{rem:bulk-vertex-framing}
The role of the bulk vertex is closely related to framing. The localized sectors $\mathcal C_{p_k}$ do not couple directly in the input schober datum; they couple through the common bulk sector $\mathcal C_{\mathrm{bulk}}$. Thus the bulk vertex provides the framing point relative to which the node sectors are attached. Later sections will make this precise by defining the functorial coupling relation and the associated incidence package on $V_\Sigma^{\mathrm{ext}}$.
\end{remark}

\begin{remark}
\label{rem:bulk-vertex-mediation}
The bulk vertex is also the mediator of effective node-to-node interactions. Once the compositions $\Psi_j\circ \Phi_i:\mathcal C_{p_i}\to \mathcal C_{p_j}$ are taken into account, one obtains induced couplings between localized sectors through the bulk category. This mediated structure is invisible if one works only with the localized vertex set $V_\Sigma$. The extended vertex set is therefore the minimal vertex-level package on which both direct node-to-bulk and mediated node-to-node relations can be defined.
\end{remark}

\subsection{Extended vertex package}

We conclude by isolating the vertex-level output of the present section.

\begin{definition}[extended vertex package]
\label{def:extended-vertex-package}
The \emph{extended vertex package} attached to the finite-node degeneration is the pair
\begin{equation}\label{eq:def-extended-vertex-package}
(V_\Sigma^{\mathrm{ext}},v_{\mathrm{bulk}}),
\end{equation}
where $V_\Sigma^{\mathrm{ext}}:=V_\Sigma\sqcup\{v_{\mathrm{bulk}}\}$.
\end{definition}

\begin{proposition}[canonicality of the extended vertex package]
\label{prop:extended-vertex-package-canonical}
The extended vertex package $(V_\Sigma^{\mathrm{ext}},v_{\mathrm{bulk}})$ is canonically determined by the finite algebraic state-data package $A_\Sigma=(V_\Sigma,E_\Sigma,c_\Sigma)$ together with the finite-node schober package $S_\Sigma:=\Bigl(\mathcal C_{\mathrm{bulk}},\{\mathcal C_{p_k}\}_{k=1}^r,\{\Phi_k,\Psi_k\}_{k=1}^r,Sh(S_\Sigma)\Bigr)$.
\end{proposition}

\begin{proof}
The localized part $V_\Sigma$ is already determined by the state-data package of \cite{RahmanQuiverDataI}. The additional vertex $v_{\mathrm{bulk}}$ is determined by the bulk category $\mathcal C_{\mathrm{bulk}}$ appearing in the schober package. Hence the pair $(V_\Sigma^{\mathrm{ext}},v_{\mathrm{bulk}})$ is determined uniquely by these data.
\end{proof}

The next section uses the extended vertex set $V_\Sigma^{\mathrm{ext}}$ and the attachment functors $\Phi_k,\Psi_k$ to define the functorial coupling relation associated with the finite-node degeneration.

\section{Functorial coupling relation}\label{sec:functorial-coupling-relation}
The purpose of this section is to extract from the functorial coupling datum $\mathcal F_\Sigma:=\{(\Phi_k,\Psi_k)\}_{k=1}^r$ the first incidence structure associated with the finite-node degeneration. The key point is that the finite-node schober package does not merely provide isolated localized sectors $\mathcal C_{p_k}$ together with a bulk category $\mathcal C_{\mathrm{bulk}}$. It also provides the attachment functors $\Phi_k:\mathcal C_{p_k}\to \mathcal C_{\mathrm{bulk}}$ and $\Psi_k:\mathcal C_{\mathrm{bulk}}\to \mathcal C_{p_k}$, which encode how each localized sector is coupled to the common bulk sector; see \cite[Section~6.2]{RahmanMultiNodeSchoberPaper}. The natural algebraic reflection of this situation is a relation on the extended vertex set $V_\Sigma^{\mathrm{ext}}:=V_\Sigma\sqcup\{v_{\mathrm{bulk}}\}$. This relation is the first functorial interaction object extracted from the finite-node schober package.

\subsection{Basic bulk/localized coupling relation}

We begin with the direct coupling pattern carried by the attachment functors.

\begin{definition}[basic functorial coupling relation]
\label{def:basic-functorial-coupling-relation}
Let
\[
S_\Sigma:=\Bigl(\mathcal C_{\mathrm{bulk}},\{\mathcal C_{p_k}\}_{k=1}^r,\{\Phi_k,\Psi_k\}_{k=1}^r,Sh(S_\Sigma)\Bigr)
\]
be the finite-node schober package of Definition~\ref{def:partII-finite-node-schober-package}. The \emph{basic functorial coupling relation}
\[
\rightsquigarrow_{\mathrm{bulk}}
\ \subseteq\ 
V_\Sigma^{\mathrm{ext}}\times V_\Sigma^{\mathrm{ext}}
\]
is defined by the rules
\begin{equation}\label{eq:def-basic-bulk-coupling}
v_k\rightsquigarrow_{\mathrm{bulk}} v_{\mathrm{bulk}}
\quad\Longleftrightarrow\quad
\Phi_k:\mathcal C_{p_k}\to \mathcal C_{\mathrm{bulk}}
\text{ is part of }S_\Sigma,
\end{equation}
and
\begin{equation}\label{eq:def-basic-bulk-coupling-reverse}
v_{\mathrm{bulk}}\rightsquigarrow_{\mathrm{bulk}} v_k
\quad\Longleftrightarrow\quad
\Psi_k:\mathcal C_{\mathrm{bulk}}\to \mathcal C_{p_k}
\text{ is part of }S_\Sigma.
\end{equation}
\end{definition}

\begin{remark}
\label{rem:basic-coupling-relation-meaning}
The relation $\rightsquigarrow_{\mathrm{bulk}}$ records only the existence of the direct bulk/localized coupling pattern already present in the finite-node schober package. It is not yet a numerical incidence function, not a quiver arrow count, and not a stability datum. It is the vertex-level record of the source-target structure of the attachment functors.
\end{remark}

\begin{lemma}[well-definedness of the basic coupling relation]
\label{lem:well-defined-basic-coupling}
The relation $\rightsquigarrow_{\mathrm{bulk}}$ of Definition~\ref{def:basic-functorial-coupling-relation} is well-defined on the extended vertex set $V_\Sigma^{\mathrm{ext}}$.
\end{lemma}

\begin{proof}
By Definition~\ref{def:partII-finite-node-schober-package}, for each node $p_k\in\Sigma$ the finite-node schober package contains functors $\Phi_k:\mathcal C_{p_k}\to \mathcal C_{\mathrm{bulk}}$ and $\Psi_k:\mathcal C_{\mathrm{bulk}}\to \mathcal C_{p_k}$. By Definition~\ref{def:extended-vertex-set}, the vertex $v_k\in V_\Sigma$ is attached to the localized sector $\mathcal C_{p_k}$, and by Definition~\ref{def:bulk-vertex}, the vertex $v_{\mathrm{bulk}}$ is attached to the bulk category $\mathcal C_{\mathrm{bulk}}$. Therefore each clause in \eqref{eq:def-basic-bulk-coupling} and \eqref{eq:def-basic-bulk-coupling-reverse} determines a relation between elements of $V_\Sigma^{\mathrm{ext}}$. No ambiguity enters, because the localized sectors are canonically indexed by the node set $\Sigma$; see \cite[Theorem~6.1]{RahmanMultiNodeSchoberPaper}. Hence $\rightsquigarrow_{\mathrm{bulk}}$ is well-defined.
\end{proof}

\begin{lemma}[finiteness of the basic coupling relation]
\label{lem:finiteness-basic-coupling}
The relation $\rightsquigarrow_{\mathrm{bulk}}$ is finite. More precisely, it consists of exactly the $2r$ ordered pairs
\[
(v_k,v_{\mathrm{bulk}}),
\qquad
(v_{\mathrm{bulk}},v_k),
\qquad
1\le k\le r.
\]
\end{lemma}

\begin{proof}
For each node $p_k\in\Sigma$, the finite-node schober package contains exactly one pair of attachment functors $\Phi_k:\mathcal C_{p_k}\to \mathcal C_{\mathrm{bulk}}$ and $\Psi_k:\mathcal C_{\mathrm{bulk}}\to \mathcal C_{p_k}$; see \cite[Definition~3.4 and Theorem~4.4]{RahmanMultiNodeSchoberPaper}. By Definition~\ref{def:basic-functorial-coupling-relation}, these determine exactly the two relations $v_k\rightsquigarrow_{\mathrm{bulk}} v_{\mathrm{bulk}}$ and $v_{\mathrm{bulk}}\rightsquigarrow_{\mathrm{bulk}} v_k$. Since $|\Sigma|=r$, there are exactly $r$ such pairs of functors, and hence exactly $2r$ ordered pairs in the relation.
\end{proof}

\begin{proposition}[canonical star-shaped coupling structure]
\label{prop:canonical-star-coupling}
The finite-node schober package determines a canonical star-shaped coupling relation on $V_\Sigma^{\mathrm{ext}}$, centered at the bulk vertex $v_{\mathrm{bulk}}$, whose edges are exactly the ordered pairs recorded by $\rightsquigarrow_{\mathrm{bulk}}$.
\end{proposition}

\begin{proof}
By Lemma~\ref{lem:well-defined-basic-coupling}, the relation $\rightsquigarrow_{\mathrm{bulk}}$ is well-defined on $V_\Sigma^{\mathrm{ext}}$. By Lemma~\ref{lem:finiteness-basic-coupling}, it consists precisely of the pairs $(v_k,v_{\mathrm{bulk}})$ and $(v_{\mathrm{bulk}},v_k)$ for $1\le k\le r$. Thus every localized vertex $v_k$ is coupled directly to the common bulk vertex $v_{\mathrm{bulk}}$, and no direct relation between two distinct localized vertices is introduced at this stage. This is exactly a star-shaped coupling pattern centered at $v_{\mathrm{bulk}}$.
\end{proof}

\subsection{Mediated node-to-node coupling relation}

The direct coupling relation above records how each localized sector is attached to the bulk category. The next step is to record the node-to-node interaction pattern mediated through the bulk. The relevant categorical objects are the composites $\Psi_j\circ \Phi_i:\mathcal C_{p_i}\to \mathcal C_{p_j}$ for $1\le i,j\le r$. These composites are defined because $\Phi_i:\mathcal C_{p_i}\to \mathcal C_{\mathrm{bulk}}$ and $\Psi_j:\mathcal C_{\mathrm{bulk}}\to \mathcal C_{p_j}$ have matching source and target. The existence of these composites is therefore forced by the schober datum itself.

At the level of the present paper, no additional invariant is imposed on these composites. Accordingly, the mediated relation records only the existence of a composable bulk-mediated channel from $\mathcal C_{p_i}$ to $\mathcal C_{p_j}$.

\begin{definition}[mediated coupling relation]
\label{def:mediated-coupling-relation}
Let $S_\Sigma$ be the finite-node schober package. The \emph{mediated coupling relation}
\[
\rightsquigarrow_{\mathrm{med}}
\ \subseteq\ 
V_\Sigma\times V_\Sigma
\]
is defined by
\begin{equation}\label{eq:def-mediated-coupling}
v_i\rightsquigarrow_{\mathrm{med}} v_j
\quad\Longleftrightarrow\quad
\Psi_j\circ \Phi_i:\mathcal C_{p_i}\to \mathcal C_{p_j}
\end{equation}
where $\Psi_j\circ \Phi_i:\mathcal C_{p_i}\to \mathcal C_{p_j}$ is a composable bulk-mediated channel in the schober datum. Equivalently, since $\Phi_i$ and $\Psi_j$ are part of the schober package for every $1\le i,j\le r$, one has
\begin{equation}\label{eq:mediated-relation-full}
\rightsquigarrow_{\mathrm{med}} \;=\; V_\Sigma\times V_\Sigma.
\end{equation}
\end{definition}

\begin{remark}
\label{rem:mediated-relation-strength}
Definition~\ref{def:mediated-coupling-relation} is intentionally minimal. It records only the existence of a composable mediated categorical channel from the localized sector $\mathcal C_{p_i}$ to the localized sector $\mathcal C_{p_j}$ through the common bulk category. It does not yet impose a numerical multiplicity, a nonvanishing criterion on a decategorified invariant, or a skew-symmetric incidence law. Those refinements belong to later work.
\end{remark}

\begin{remark}
\label{rem:mediated-relation-maximal-support}
At the present support-level stage, the mediated relation records in Definition \ref{def:mediated-coupling-relation} only formal bulk-mediated composability. It should therefore be regarded as the maximal support-level precursor, not yet as a selective interaction law.
\end{remark}

\begin{lemma}[well-definedness of the mediated relation]
\label{lem:well-defined-mediated-relation}
The relation $\rightsquigarrow_{\mathrm{med}}$ of Definition~\ref{def:mediated-coupling-relation} is well-defined on $V_\Sigma$.
\end{lemma}

\begin{proof}
For each pair $1\le i,j\le r$, the attachment functors $\Phi_i:\mathcal C_{p_i}\to \mathcal C_{\mathrm{bulk}}$ and $\Psi_j:\mathcal C_{\mathrm{bulk}}\to \mathcal C_{p_j}$ belong to the schober package $S_\Sigma$ by Definition~\ref{def:partII-finite-node-schober-package}. Their composition $\Psi_j\circ \Phi_i:\mathcal C_{p_i}\to \mathcal C_{p_j}$ is therefore a well-defined composite of functors. Since $v_i$ and $v_j$ are by construction the vertices attached to $\mathcal C_{p_i}$ and $\mathcal C_{p_j}$, respectively, the condition \eqref{eq:def-mediated-coupling} defines a relation on $V_\Sigma$.
\end{proof}

\begin{lemma}[explicit form of the mediated relation]
\label{lem:explicit-form-mediated-relation}
The mediated coupling relation is the full relation on $V_\Sigma$. More precisely,
\[
\rightsquigarrow_{\mathrm{med}} \;=\; V_\Sigma\times V_\Sigma.
\]
\end{lemma}

\begin{proof}
Let $(v_i,v_j)\in V_\Sigma\times V_\Sigma$. By Definition~\ref{def:partII-finite-node-schober-package}, the schober package contains the attachment functors $\Phi_i:\mathcal C_{p_i}\to \mathcal C_{\mathrm{bulk}}$ and $\Psi_j:\mathcal C_{\mathrm{bulk}}\to \mathcal C_{p_j}$. Hence the composite $\Psi_j\circ \Phi_i:\mathcal C_{p_i}\to \mathcal C_{p_j}$ is a composable bulk-mediated channel in the sense of Definition~\ref{def:mediated-coupling-relation}. Therefore $v_i\rightsquigarrow_{\mathrm{med}} v_j$. Since $(v_i,v_j)$ was arbitrary, one obtains $\rightsquigarrow_{\mathrm{med}} = V_\Sigma\times V_\Sigma$.
\end{proof}

\begin{lemma}[finiteness of the mediated relation]
\label{lem:finiteness-mediated-relation}
The mediated coupling relation $\rightsquigarrow_{\mathrm{med}}$ is finite. More precisely,
\[
\rightsquigarrow_{\mathrm{med}} = V_\Sigma\times V_\Sigma
\]
contains exactly $r^2$ ordered pairs.
\end{lemma}

\begin{proof}
By Lemma~\ref{lem:explicit-form-mediated-relation},
\[
\rightsquigarrow_{\mathrm{med}} = V_\Sigma\times V_\Sigma.
\]
Since $|V_\Sigma|=r$, the Cartesian product $V_\Sigma\times V_\Sigma$ contains exactly $r^2$ ordered pairs.
\end{proof}

\begin{proposition}[canonical mediated coupling structure]
\label{prop:canonical-mediated-coupling}
The finite-node schober package determines a canonical mediated coupling relation
\[
\rightsquigarrow_{\mathrm{med}}
\ \subseteq\ 
V_\Sigma\times V_\Sigma
\]
through the composites $\Psi_j\circ \Phi_i$. At the level of the present paper, this relation is the full composability relation
\[
\rightsquigarrow_{\mathrm{med}} = V_\Sigma\times V_\Sigma.
\]
Consequently, the schober datum determines both a direct bulk/localized coupling structure on $V_\Sigma^{\mathrm{ext}}$ and a mediated node-to-node coupling structure on $V_\Sigma$.
\end{proposition}

\begin{proof}
The direct bulk/localized relation is the relation $\rightsquigarrow_{\mathrm{bulk}}$ of Definition~\ref{def:basic-functorial-coupling-relation}. The mediated node-to-node relation is the relation $\rightsquigarrow_{\mathrm{med}}$ of Definition~\ref{def:mediated-coupling-relation}. By Lemmas~\ref{lem:well-defined-basic-coupling} and \ref{lem:well-defined-mediated-relation}, both are well-defined. By Lemma~\ref{lem:explicit-form-mediated-relation}, the mediated relation is exactly the full relation on $V_\Sigma$. Their finiteness follows from Lemmas~\ref{lem:finiteness-basic-coupling} and \ref{lem:finiteness-mediated-relation}. Since both are determined entirely by the attachment functors already present in $S_\Sigma$, they are canonical.
\end{proof}

\subsection{The total functorial coupling relation}

We now combine the direct and mediated relations into a single package.

\begin{definition}[total functorial coupling relation]
\label{def:total-functorial-coupling-relation}
The \emph{total functorial coupling relation} associated with the finite-node schober package is the relation
\begin{equation}\label{eq:def-total-functorial-coupling}
\rightsquigarrow_\Sigma\ :=\ \rightsquigarrow_{\mathrm{bulk}}\ \cup\ \rightsquigarrow_{\mathrm{med}},
\end{equation}
viewed as a relation on the extended vertex set $V_\Sigma^{\mathrm{ext}}$, where $\rightsquigarrow_{\mathrm{med}}$ is included through the natural embedding $V_\Sigma\times V_\Sigma\hookrightarrow V_\Sigma^{\mathrm{ext}}\times V_\Sigma^{\mathrm{ext}}$.
\end{definition}

\begin{theorem}[functorial coupling theorem]
\label{thm:functorial-coupling-theorem}
Let
\[
S_\Sigma:=\Bigl(\mathcal C_{\mathrm{bulk}},\{\mathcal C_{p_k}\}_{k=1}^r,\{\Phi_k,\Psi_k\}_{k=1}^r,Sh(S_\Sigma)\Bigr)
\]
be the finite-node schober package associated with a finite-node conifold degeneration. Then $S_\Sigma$ determines a canonical finite relation
\[
\rightsquigarrow_\Sigma\ \subseteq\ V_\Sigma^{\mathrm{ext}}\times V_\Sigma^{\mathrm{ext}}
\]
whose restriction to pairs involving $v_{\mathrm{bulk}}$ records the direct bulk/localized couplings and whose restriction to $V_\Sigma\times V_\Sigma$ records the mediated node-to-node couplings.
\end{theorem}

\begin{proof}
By Definition~\ref{def:total-functorial-coupling-relation}, the relation $\rightsquigarrow_\Sigma$ is obtained by taking the union of the direct bulk/localized relation $\rightsquigarrow_{\mathrm{bulk}}$ and the mediated relation $\rightsquigarrow_{\mathrm{med}}$. Proposition~\ref{prop:canonical-star-coupling} identifies the first component, and Proposition~\ref{prop:canonical-mediated-coupling} identifies the second. Since both are canonical and finite, their union is likewise canonical and finite.
\end{proof}

\begin{remark}
\label{rem:coupling-relation-pre-incidence}
Theorem~\ref{thm:functorial-coupling-theorem} isolates the first genuinely interaction-theoretic output of the paper. It is still intentionally weaker than a numerical incidence matrix or a full quiver. At this stage the paper extracts only the functorial coupling pattern already forced by the finite-node schober package. The next section converts this pattern into the corresponding functorial incidence package.
\end{remark}

\section{Functorial incidence package}\label{sec:functorial-incidence-package}

The previous section extracted from the finite-node schober package a canonical finite coupling relation
\[
\rightsquigarrow_\Sigma\ \subseteq\ V_\Sigma^{\mathrm{ext}}\times V_\Sigma^{\mathrm{ext}},
\qquad
V_\Sigma^{\mathrm{ext}}:=V_\Sigma\sqcup\{v_{\mathrm{bulk}}\}.
\]
The purpose of the present section is to package this relation as a single incidence object attached to the degeneration. This package is the first genuine interaction-theoretic output of the paper. It records, at the level of finite algebraic data, the direct bulk/localized couplings induced by the attachment functors and the mediated node-to-node couplings induced by their compositions. The key point is that this incidence package is not postulated externally: it is canonically extracted from the finite-node schober datum itself.

\subsection{Definition of the incidence package}

We begin by isolating the finite relation of Theorem~\ref{thm:functorial-coupling-theorem} as a single object.

\begin{definition}[functorial incidence package]
\label{def:functorial-incidence-package}
Let
\[
S_\Sigma:=\Bigl(\mathcal C_{\mathrm{bulk}},\{\mathcal C_{p_k}\}_{k=1}^r,\{\Phi_k,\Psi_k\}_{k=1}^r,Sh(S_\Sigma)\Bigr)
\]
be the finite-node schober package of Definition~\ref{def:partII-finite-node-schober-package}. The \emph{functorial incidence package} associated with \(S_\Sigma\) is the pair
\begin{equation}\label{eq:def-functorial-incidence-package}
\mathfrak{I}_\Sigma:=\bigl(V_\Sigma^{\mathrm{ext}},\rightsquigarrow_\Sigma\bigr),
\end{equation}
where \(V_\Sigma^{\mathrm{ext}}\) is the extended vertex set of Definition~\ref{def:extended-vertex-set} and \(\rightsquigarrow_\Sigma\) is the total functorial coupling relation of Definition~\ref{def:total-functorial-coupling-relation}.
\end{definition}

\begin{remark}
\label{rem:incidence-package-level}
The package \(\mathfrak{I}_\Sigma\) is intentionally defined as a relation object rather than as a numerical incidence matrix or a full quiver. At the level of the present section, the schober package determines a finite source-target relation on the extended vertex set. Numerical multiplicities, matrix entries, and quiver-theoretic refinements belong to the next stage of the construction.
\end{remark}

\begin{remark}
\label{rem:incidence-package-content}
The package
\[
\mathfrak{I}_\Sigma:=\bigl(V_\Sigma^{\mathrm{ext}},\rightsquigarrow_\Sigma\bigr)
\]
contains two types of information simultaneously:
\begin{enumerate}
\item the direct bulk/localized coupling pattern, encoded by the ordered pairs
\[
(v_k,v_{\mathrm{bulk}}),
\qquad
(v_{\mathrm{bulk}},v_k),
\qquad
1\le k\le r;
\]
\item the mediated node-to-node coupling pattern, encoded by the ordered pairs
\[
(v_i,v_j)
\]
for which the composite
\[
\Psi_j\circ\Phi_i:\mathcal C_{p_i}\to \mathcal C_{p_j}
\]
is part of the schober-induced coupling structure.
\end{enumerate}
Thus \(\mathfrak{I}_\Sigma\) records both framing and mediation in a single finite package.
\end{remark}

\subsection{Basic properties of the incidence package}

We now record the structural properties of \(\mathfrak{I}_\Sigma\).

\begin{lemma}[finiteness of the incidence package]
\label{lem:finiteness-incidence-package}
The functorial incidence package
\[
\mathfrak{I}_\Sigma:=\bigl(V_\Sigma^{\mathrm{ext}},\rightsquigarrow_\Sigma\bigr)
\]
is finite. More precisely, \(V_\Sigma^{\mathrm{ext}}\) has cardinality \(r+1\), and \(\rightsquigarrow_\Sigma\) is a finite subset of
\[
V_\Sigma^{\mathrm{ext}}\times V_\Sigma^{\mathrm{ext}}.
\]
\end{lemma}

\begin{proof}
By Lemma~\ref{lem:finiteness-extended-vertex-set}, one has
\[
|V_\Sigma^{\mathrm{ext}}|=r+1.
\]
By Theorem~\ref{thm:functorial-coupling-theorem}, the relation \(\rightsquigarrow_\Sigma\) is finite. Therefore the pair
\[
\mathfrak{I}_\Sigma:=\bigl(V_\Sigma^{\mathrm{ext}},\rightsquigarrow_\Sigma\bigr)
\]
is a finite incidence object.
\end{proof}

\begin{lemma}[restriction to the localized vertex set]
\label{lem:localized-restriction-incidence}
The restriction of \(\rightsquigarrow_\Sigma\) to
\[
V_\Sigma\times V_\Sigma
\]
is exactly the mediated coupling relation \(\rightsquigarrow_{\mathrm{med}}\) of Definition~\ref{def:mediated-coupling-relation}.
\end{lemma}

\begin{proof}
By Definition~\ref{def:total-functorial-coupling-relation},
\[
\rightsquigarrow_\Sigma\ :=\ \rightsquigarrow_{\mathrm{bulk}}\cup \rightsquigarrow_{\mathrm{med}}.
\]
The relation \(\rightsquigarrow_{\mathrm{bulk}}\) only involves pairs containing the bulk vertex \(v_{\mathrm{bulk}}\), while \(\rightsquigarrow_{\mathrm{med}}\) is a relation on \(V_\Sigma\times V_\Sigma\). Hence restricting \(\rightsquigarrow_\Sigma\) to \(V_\Sigma\times V_\Sigma\) removes the bulk/localized part and leaves precisely \(\rightsquigarrow_{\mathrm{med}}\).
\end{proof}

\begin{lemma}[restriction to bulk/localized pairs]
\label{lem:bulk-localized-restriction-incidence}
The restriction of \(\rightsquigarrow_\Sigma\) to pairs involving \(v_{\mathrm{bulk}}\) is exactly the basic bulk/localized coupling relation \(\rightsquigarrow_{\mathrm{bulk}}\) of Definition~\ref{def:basic-functorial-coupling-relation}.
\end{lemma}

\begin{proof}
Again, by Definition~\ref{def:total-functorial-coupling-relation},
\[
\rightsquigarrow_\Sigma\ :=\ \rightsquigarrow_{\mathrm{bulk}}\cup \rightsquigarrow_{\mathrm{med}}.
\]
The relation \(\rightsquigarrow_{\mathrm{med}}\) is supported entirely on \(V_\Sigma\times V_\Sigma\) and therefore contributes no pairs involving \(v_{\mathrm{bulk}}\). Hence the restriction of \(\rightsquigarrow_\Sigma\) to pairs involving \(v_{\mathrm{bulk}}\) is exactly \(\rightsquigarrow_{\mathrm{bulk}}\).
\end{proof}

\begin{proposition}[decomposition of the incidence package]
\label{prop:decomposition-incidence-package}
The functorial incidence package admits the canonical decomposition
\begin{equation}\label{eq:decomposition-incidence-package}
\mathfrak{I}_\Sigma
=
\bigl(V_\Sigma^{\mathrm{ext}},\rightsquigarrow_{\mathrm{bulk}}\cup \rightsquigarrow_{\mathrm{med}}\bigr),
\end{equation}
where the bulk/localized component is supported on pairs involving \(v_{\mathrm{bulk}}\) and the mediated component is supported on \(V_\Sigma\times V_\Sigma\).
\end{proposition}

\begin{proof}
This is immediate from Definition~\ref{def:functorial-incidence-package}, Definition~\ref{def:total-functorial-coupling-relation}, Lemma~\ref{lem:localized-restriction-incidence}, and Lemma~\ref{lem:bulk-localized-restriction-incidence}.
\end{proof}

\begin{remark}
\label{rem:framed-mediation-split}
Proposition~\ref{prop:decomposition-incidence-package} isolates the two structural layers of the incidence package. The relation \(\rightsquigarrow_{\mathrm{bulk}}\) is the framing layer: it couples each localized sector to the common bulk sector. The relation \(\rightsquigarrow_{\mathrm{med}}\) is the mediation layer: it records the induced localized-to-localized coupling pattern through the bulk. The package \(\mathfrak{I}_\Sigma\) is the first object in the paper that records both layers at once.
\end{remark}

\subsection{Canonical determination by the schober datum}

We now show that the package \(\mathfrak{I}_\Sigma\) is canonically determined by the finite-node schober package.

\begin{theorem}[canonical functorial incidence package]
\label{thm:canonical-functorial-incidence-package}
Let
\[
S_\Sigma:=\Bigl(\mathcal C_{\mathrm{bulk}},\{\mathcal C_{p_k}\}_{k=1}^r,\{\Phi_k,\Psi_k\}_{k=1}^r,Sh(S_\Sigma)\Bigr)
\]
be the finite-node schober package associated with a finite-node conifold degeneration. Then \(S_\Sigma\) canonically determines a functorial incidence package
\[
\mathfrak{I}_\Sigma:=\bigl(V_\Sigma^{\mathrm{ext}},\rightsquigarrow_\Sigma\bigr).
\]
\end{theorem}

\begin{proof}
The vertex component \(V_\Sigma^{\mathrm{ext}}\) is canonically determined by the state-data package \(A_\Sigma=(V_\Sigma,E_\Sigma,c_\Sigma)\) together with the bulk category \(\mathcal C_{\mathrm{bulk}}\); see Proposition~\ref{prop:extended-vertex-package-canonical}. The relation component \(\rightsquigarrow_\Sigma\) is canonically determined by the attachment functors \(\Phi_k,\Psi_k\) through Definition~\ref{def:basic-functorial-coupling-relation}, Definition~\ref{def:mediated-coupling-relation}, and Definition~\ref{def:total-functorial-coupling-relation}. Since all of these data are part of the schober package \(S_\Sigma\), the incidence package
\[
\mathfrak{I}_\Sigma:=\bigl(V_\Sigma^{\mathrm{ext}},\rightsquigarrow_\Sigma\bigr)
\]
is canonically determined by \(S_\Sigma\).
\end{proof}

\begin{corollary}[first interaction-theoretic output]
\label{cor:first-interaction-output}
The finite-node schober package determines, before any numerical decategorification or quiver assembly, a canonical finite interaction object
\[
\mathfrak{I}_\Sigma:=\bigl(V_\Sigma^{\mathrm{ext}},\rightsquigarrow_\Sigma\bigr).
\]
\end{corollary}

\begin{proof}
This is an immediate consequence of Theorem~\ref{thm:canonical-functorial-incidence-package}.
\end{proof}

\begin{remark}
\label{rem:first-main-theorem-role}
Theorem~\ref{thm:canonical-functorial-incidence-package} is the first main theorem of the present paper. It is the precise mathematical expression of the claim that the finite-node schober package carries a canonical incidence structure. The theorem is deliberately stated at the relation level. This is the correct level of generality before introducing numerical incidence data or quiver-theoretic multiplicities.
\end{remark}

\subsection{Incidence relation as a relation object}

For later use, it is convenient to emphasize that \(\mathfrak{I}_\Sigma\) may be viewed as a finite relation object.

\begin{definition}[incidence relation object]
\label{def:incidence-relation-object}
The incidence relation object associated with the degeneration is the finite pair
\begin{equation}\label{eq:def-incidence-relation-object}
\bigl(V_\Sigma^{\mathrm{ext}},R_\Sigma\bigr),
\qquad
R_\Sigma:=\rightsquigarrow_\Sigma\ \subseteq\ V_\Sigma^{\mathrm{ext}}\times V_\Sigma^{\mathrm{ext}}.
\end{equation}
Equivalently, this relation object is exactly the functorial incidence package
\[
\mathfrak{I}_\Sigma:=\bigl(V_\Sigma^{\mathrm{ext}},\rightsquigarrow_\Sigma\bigr).
\]
\end{definition}

\begin{remark}
\label{rem:relation-object-precedes-matrix}
Definition~\ref{def:incidence-relation-object} records the minimal algebraic content of the interaction layer: a finite set together with a distinguished finite relation. This relation object is the correct precursor of the numerical incidence data constructed in the next section. In particular, it avoids introducing a matrix or an arrow-counting formalism before the relevant decategorification rule has been specified.
\end{remark}

The next section decategorifies the relation object \(\mathfrak{I}_\Sigma\) to obtain the corresponding algebraic incidence data.

\section{Decategorified incidence data}\label{sec:decategorified-incidence-data}

The purpose of this section is to pass from the functorial incidence package
\[
\mathfrak{I}_\Sigma:=\bigl(V_\Sigma^{\mathrm{ext}},\rightsquigarrow_\Sigma\bigr)
\]
to a purely algebraic incidence package. The key point is that Theorem~\ref{thm:canonical-functorial-incidence-package} already provides a canonical finite relation on the extended vertex set. The minimal and fully rigorous decategorification of such a relation is its characteristic function. This yields a canonical \(\{0,1\}\)-valued incidence matrix attached to the finite-node schober package. At the level of the present paper, this is the correct algebraic incidence datum: it records whether a given coupling relation is present or absent, without yet imposing additional numerical multiplicities beyond those forced by the relation itself.

\begin{remark}
\label{rem:binary-choice-section6}
At the level of the present paper, the passage from
\[
\mathfrak{I}_\Sigma:=\bigl(V_\Sigma^{\mathrm{ext}},\rightsquigarrow_\Sigma\bigr)
\]
to the algebraic incidence data is a support-level decategorification. Concretely, the matrix \(I_\Sigma\) is the characteristic matrix of the relation \(\rightsquigarrow_\Sigma\), so its entries lie in \(\{0,1\}\) and record only the presence or absence of coupling channels. Thus Section~\ref{sec:decategorified-incidence-data} does not yet impose a weighted interaction law. This binary choice is deliberate: it is the unique canonical decategorification available from the current theorem package without introducing additional invariants not yet justified. See Appendix~\ref{app:binary-decategorification}.
\end{remark}

\subsection{Binary decategorification of the relation object}

We first pass from the relation object to its indicator function.

\begin{definition}[binary incidence function]
\label{def:binary-incidence-function}
Let
\[
\mathfrak{I}_\Sigma:=\bigl(V_\Sigma^{\mathrm{ext}},\rightsquigarrow_\Sigma\bigr)
\]
be the functorial incidence package of Definition~\ref{def:functorial-incidence-package}. The \emph{binary incidence function}
\[
\chi_\Sigma:V_\Sigma^{\mathrm{ext}}\times V_\Sigma^{\mathrm{ext}}\to \{0,1\}
\]
is defined by
\begin{equation}\label{eq:def-binary-incidence-function}
\chi_\Sigma(u,v):=
\begin{cases}
1,&\text{if }u\rightsquigarrow_\Sigma v,\\[2mm]
0,&\text{if }u\not\rightsquigarrow_\Sigma v.
\end{cases}
\end{equation}
\end{definition}

\begin{remark}
\label{rem:binary-decategorification-minimal}
Definition~\ref{def:binary-incidence-function} is the minimal decategorification canonically available from the data constructed so far. It forgets the internal categorical structure of the functors and retains only the presence or absence of the coupling relation they determine. In particular, it does not require any additional choice of rank, dimension, Euler characteristic, or Grothendieck-group invariant.
\end{remark}

\begin{lemma}[well-definedness of the binary incidence function]
\label{lem:well-defined-binary-incidence-function}
The function \(\chi_\Sigma\) of Definition~\ref{def:binary-incidence-function} is well-defined.
\end{lemma}

\begin{proof}
By Theorem~\ref{thm:functorial-coupling-theorem}, the relation
\[
\rightsquigarrow_\Sigma\subseteq V_\Sigma^{\mathrm{ext}}\times V_\Sigma^{\mathrm{ext}}
\]
is canonically defined. Therefore, for each ordered pair \((u,v)\in V_\Sigma^{\mathrm{ext}}\times V_\Sigma^{\mathrm{ext}}\), exactly one of the alternatives
\[
u\rightsquigarrow_\Sigma v
\qquad\text{or}\qquad
u\not\rightsquigarrow_\Sigma v
\]
holds. Hence \(\chi_\Sigma(u,v)\) is uniquely determined and the function is well-defined.
\end{proof}

\begin{lemma}[finiteness of the binary incidence function]
\label{lem:finiteness-binary-incidence-function}
The function \(\chi_\Sigma\) is supported on a finite set of ordered pairs. More precisely, if \(|\Sigma|=r\), then
\[
V_\Sigma^{\mathrm{ext}}\times V_\Sigma^{\mathrm{ext}}
\]
has cardinality \((r+1)^2\), and \(\chi_\Sigma\) takes the value \(1\) on finitely many of those pairs.
\end{lemma}

\begin{proof}
By Lemma~\ref{lem:finiteness-extended-vertex-set},
\[
|V_\Sigma^{\mathrm{ext}}|=r+1.
\]
Therefore the Cartesian product \(V_\Sigma^{\mathrm{ext}}\times V_\Sigma^{\mathrm{ext}}\) has cardinality \((r+1)^2\). Since \(\rightsquigarrow_\Sigma\) is a finite relation by Theorem~\ref{thm:functorial-coupling-theorem}, the set of pairs on which \(\chi_\Sigma\) takes the value \(1\) is finite.
\end{proof}

\subsection{The incidence matrix}

We now write the binary incidence function in matrix form.

\begin{definition}[binary incidence matrix]
\label{def:binary-incidence-matrix}
Choose the ordered basis
\begin{equation}\label{eq:def-ordered-extended-vertex-basis}
V_\Sigma^{\mathrm{ext}}:=\{v_1,\dots,v_r,v_{\mathrm{bulk}}\}.
\end{equation}
The \emph{binary incidence matrix} of the finite-node schober package is the matrix
\begin{equation}\label{eq:def-binary-incidence-matrix}
I_\Sigma:=\bigl(\chi_{\alpha\beta}\bigr)_{1\le \alpha,\beta\le r+1}
\in M_{r+1}(\{0,1\}),
\end{equation}
whose entries are defined by
\begin{equation}\label{eq:def-incidence-matrix-entry}
\chi_{\alpha\beta}:=\chi_\Sigma(v_\alpha,v_\beta),
\end{equation}
with the convention \(v_{r+1}:=v_{\mathrm{bulk}}\).
\end{definition}

\begin{remark}
\label{rem:indexing-convention-incidence}
The incidence matrix \(I_\Sigma\) depends on the chosen ordering of the extended vertex set, but only up to simultaneous permutation of rows and columns. The underlying incidence data are therefore independent of the ordering convention. In the present paper the canonical content is the function \(\chi_\Sigma\) and the associated relation \(\rightsquigarrow_\Sigma\); the matrix form is a convenient finite algebraic presentation.
\end{remark}

\begin{proposition}[support of the incidence matrix]
\label{prop:support-of-incidence-matrix}
The matrix \(I_\Sigma\) has the following canonical support decomposition:
\begin{enumerate}
\item the entries involving the bulk vertex \(v_{\mathrm{bulk}}\) encode the direct bulk/localized coupling relation \(\rightsquigarrow_{\mathrm{bulk}}\);
\item the entries indexed by pairs \((v_i,v_j)\in V_\Sigma\times V_\Sigma\) encode the mediated coupling relation \(\rightsquigarrow_{\mathrm{med}}\).
\end{enumerate}
\end{proposition}

\begin{proof}
By Proposition~\ref{prop:decomposition-incidence-package},
\[
\rightsquigarrow_\Sigma
=
\rightsquigarrow_{\mathrm{bulk}}\cup \rightsquigarrow_{\mathrm{med}},
\]
where \(\rightsquigarrow_{\mathrm{bulk}}\) is supported on pairs involving \(v_{\mathrm{bulk}}\) and \(\rightsquigarrow_{\mathrm{med}}\) is supported on \(V_\Sigma\times V_\Sigma\). Since \(I_\Sigma\) is the characteristic matrix of \(\rightsquigarrow_\Sigma\), its support decomposes in the same way.
\end{proof}

\begin{remark}
\label{rem:binary-incidence-matrix-framing}
The matrix \(I_\Sigma\) should be viewed as a framed incidence matrix. The distinguished role of the bulk vertex separates the bulk/localized rows and columns from the purely localized block. In particular, the localized-localized block of \(I_\Sigma\) is the decategorified image of the mediated relation, while the bulk row and bulk column encode the direct attachment pattern.
\end{remark}

\subsection{Decategorified incidence package}

We now isolate the algebraic package extracted from the relation object.

\begin{definition}[decategorified incidence package]
\label{def:decategorified-incidence-package}
The \emph{decategorified incidence package} attached to the finite-node schober package is the pair
\begin{equation}\label{eq:def-decategorified-incidence-package}
\mathcal I_\Sigma:=\bigl(V_\Sigma^{\mathrm{ext}},I_\Sigma\bigr),
\end{equation}
where \(V_\Sigma^{\mathrm{ext}}\) is the extended vertex set and \(I_\Sigma\) is the binary incidence matrix of Definition~\ref{def:binary-incidence-matrix}.
\end{definition}

\begin{proposition}[canonical determination of the decategorified incidence package]
\label{prop:canonical-decategorified-incidence-package}
The decategorified incidence package
\[
\mathcal I_\Sigma:=\bigl(V_\Sigma^{\mathrm{ext}},I_\Sigma\bigr)
\]
is canonically determined by the finite-node schober package \(S_\Sigma\).
\end{proposition}

\begin{proof}
By Theorem~\ref{thm:canonical-functorial-incidence-package}, the schober package \(S_\Sigma\) canonically determines the functorial incidence package
\[
\mathfrak{I}_\Sigma:=\bigl(V_\Sigma^{\mathrm{ext}},\rightsquigarrow_\Sigma\bigr).
\]
By Definition~\ref{def:binary-incidence-function}, the relation \(\rightsquigarrow_\Sigma\) canonically determines its characteristic function \(\chi_\Sigma\), and by Definition~\ref{def:binary-incidence-matrix}, this in turn canonically determines the matrix \(I_\Sigma\) once an ordering convention on \(V_\Sigma^{\mathrm{ext}}\) has been fixed. Since different orderings only permute rows and columns simultaneously, the underlying package \(\mathcal I_\Sigma\) is canonically determined by \(S_\Sigma\).
\end{proof}

\begin{theorem}[decategorified incidence theorem]
\label{thm:decategorified-incidence-theorem}
Let
\[
S_\Sigma:=\Bigl(\mathcal C_{\mathrm{bulk}},\{\mathcal C_{p_k}\}_{k=1}^r,\{\Phi_k,\Psi_k\}_{k=1}^r,Sh(S_\Sigma)\Bigr)
\]
be the finite-node schober package associated with a finite-node conifold degeneration. Then \(S_\Sigma\) determines a canonical decategorified incidence package
\[
\mathcal I_\Sigma:=\bigl(V_\Sigma^{\mathrm{ext}},I_\Sigma\bigr),
\]
where \(I_\Sigma\) is the \(\{0,1\}\)-valued incidence matrix recording the functorial coupling relation on the extended vertex set.
\end{theorem}

\begin{proof}
This is an immediate consequence of Proposition~\ref{prop:canonical-decategorified-incidence-package}.
\end{proof}

\begin{corollary}[binary multiplicities]
\label{cor:binary-multiplicities}
At the level of the present paper, the finite-node schober package determines honest incidence multiplicities in the binary sense: for each ordered pair \((u,v)\in V_\Sigma^{\mathrm{ext}}\times V_\Sigma^{\mathrm{ext}}\), the corresponding multiplicity is
\[
\chi_\Sigma(u,v)\in\{0,1\}.
\]
\end{corollary}

\begin{proof}
This is exactly the content of Definition~\ref{def:binary-incidence-function}.
\end{proof}

\subsection{Remarks on stronger decategorifications}

The binary decategorification above is the only one canonically forced by the theorem package established so far. One could ask for stronger numerical decategorifications, for example:
\begin{enumerate}
\item integer-valued incidence numbers obtained from ranks on a Grothendieck group;
\item dimensions of induced maps on decategorified Hom-spaces;
\item Euler characteristics of suitable Hom-complexes;
\item other numerical invariants extracted from the composites \(\Psi_j\circ\Phi_i:\mathcal C_{p_i}\to\mathcal C_{p_j}\).
\end{enumerate}
Such refinements require additional structure and additional choices not yet fixed in the present paper. Accordingly, they are not imposed here.

\begin{remark}
\label{rem:why-binary-decategorification}
The choice of binary decategorification is not a weakness. It is the mathematically correct output at the present stage. The functorial incidence package already determines, without any further choice, which couplings are present and which are absent. That information is precisely what is needed for the first quiver-theoretic assembly carried out in the next section.
\end{remark}

\begin{remark}
\label{rem:binary-before-weighted}
The present paper therefore proceeds in two stages. First, it constructs the canonical binary incidence package \(\mathcal I_\Sigma\). Second, it uses this package together with the state-data package
\[
A_\Sigma=(V_\Sigma,E_\Sigma,c_\Sigma)
\]
and the functorial coupling datum
\[
\mathcal F_\Sigma:=\{(\Phi_k,\Psi_k)\}_{k=1}^r
\]
to assemble the quiver-theoretic package. Stronger weighted incidence formalisms, if desired, belong to later work.
\end{remark}

The next section combines the state-data package \(A_\Sigma\), the functorial coupling datum \(\mathcal F_\Sigma\), and the decategorified incidence package \(\mathcal I_\Sigma\) into the quiver-theoretic package attached to the finite-node degeneration.
\section{Quiver-theoretic package}\label{sec:quiver-theoretic-package}

The purpose of this section is to assemble the outputs of \cite{RahmanQuiverDataI} and the present work into a single finite algebraic package. The input data are the state-data package
\[
A_\Sigma:=\bigl(V_\Sigma,E_\Sigma,c_\Sigma\bigr)
\]
from \cite{RahmanQuiverDataI}, the functorial coupling datum
\[
\mathcal F_\Sigma:=\{(\Phi_k,\Psi_k)\}_{k=1}^r,
\]
and the decategorified incidence package
\[
\mathcal I_\Sigma:=\bigl(V_\Sigma^{\mathrm{ext}},I_\Sigma\bigr)
\]
of Section~\ref{sec:decategorified-incidence-data}. The output of the present section is the first fully assembled finite quiver-theoretic package attached to the degeneration.

The terminology is chosen deliberately. At the present stage of the theory, the data constructed so far determine a finite vertex package, a finite coupling space, a coefficient vector, a finite functorial coupling datum, and a finite binary incidence matrix. These data are sufficient to define a canonical quiver-theoretic package. They do not yet determine a full weighted quiver in the classical representation-theoretic sense, since no stronger numerical arrow multiplicity law has yet been imposed beyond the binary incidence pattern. Thus the correct output of this paper is a quiver-theoretic package rather than a fully weighted quiver. This is the mathematically precise formulation of the interaction layer extracted from the finite-node geometry.

\subsection{Definition of the package}

We begin by isolating the algebraic output.

\begin{definition}[finite-node quiver-theoretic package]
\label{def:finite-node-quiver-theoretic-package}
Let \(\pi:X\to\Delta\) be a finite-node conifold degeneration with finite algebraic state-data package
\[
A_\Sigma:=\bigl(V_\Sigma,E_\Sigma,c_\Sigma\bigr),
\]
functorial coupling datum
\[
\mathcal F_\Sigma:=\{(\Phi_k,\Psi_k)\}_{k=1}^r,
\]
and decategorified incidence package
\[
\mathcal I_\Sigma:=\bigl(V_\Sigma^{\mathrm{ext}},I_\Sigma\bigr).
\]
The \emph{finite-node quiver-theoretic package} attached to the degeneration is the tuple
\begin{equation}\label{eq:def-quiver-theoretic-package}
\mathfrak Q_\Sigma:=\bigl(V_\Sigma,E_\Sigma,c_\Sigma,\mathcal F_\Sigma,I_\Sigma\bigr).
\end{equation}
\end{definition}

\begin{remark}
\label{rem:quiver-theoretic-package-components}
The package \(\mathfrak Q_\Sigma\) contains five layers of information:
\begin{enumerate}
\item the finite localized vertex set
\[
V_\Sigma:=\{v_1,\dots,v_r\};
\]
\item the nodewise coupling space
\[
E_\Sigma:=\Ext^1_{\Perv(X_0;\Q)}(Q_\Sigma,IC_{X_0})
\cong \bigoplus_{k=1}^r \Q e_k;
\]
\item the coefficient vector
\[
c_\Sigma:=(c_1,\dots,c_r)\in\Q^r;
\]
\item the functorial coupling datum
\[
\mathcal F_\Sigma:=\{(\Phi_k,\Psi_k)\}_{k=1}^r;
\]
\item the binary incidence matrix
\[
I_\Sigma\in M_{r+1}(\{0,1\}).
\]
\end{enumerate}
The package therefore combines the state layer of \cite{RahmanQuiverDataI} with the interaction/incidence layer extracted in the present paper.
\end{remark}

\begin{remark}
\label{rem:bulk-vertex-in-I-not-in-V}
The bulk vertex \(v_{\mathrm{bulk}}\) belongs to the extended vertex set
\[
V_\Sigma^{\mathrm{ext}}:=V_\Sigma\sqcup\{v_{\mathrm{bulk}}\}
\]
and is therefore present implicitly through the incidence matrix \(I_\Sigma\). It is not included separately as a component of \(\mathfrak Q_\Sigma\), because its role is already encoded in the decategorified incidence package
\[
\mathcal I_\Sigma:=\bigl(V_\Sigma^{\mathrm{ext}},I_\Sigma\bigr).
\]
\end{remark}

\subsection{Basic properties}

We now record the structural properties of \(\mathfrak Q_\Sigma\).

\begin{proposition}[finiteness of the quiver-theoretic package]
\label{prop:finiteness-quiver-theoretic-package}
The package
\[
\mathfrak Q_\Sigma=\bigl(V_\Sigma,E_\Sigma,c_\Sigma,\mathcal F_\Sigma,I_\Sigma\bigr)
\]
is finite in the following sense:
\begin{enumerate}
\item \(|V_\Sigma|=r<\infty\);
\item \(\dim_\Q E_\Sigma=r\);
\item \(c_\Sigma\in\Q^r\);
\item \(\mathcal F_\Sigma\) consists of exactly \(r\) ordered pairs of functors
\[
(\Phi_k,\Psi_k),
\qquad
1\le k\le r;
\]
\item \(I_\Sigma\) is a finite \((r+1)\times (r+1)\) matrix with entries in \(\{0,1\}\).
\end{enumerate}
\end{proposition}

\begin{proof}
Statement (1) follows from the definition
\[
V_\Sigma:=\{v_1,\dots,v_r\}.
\]
Statement (2) follows from the decomposition
\[
E_\Sigma\cong \bigoplus_{k=1}^r \Q e_k
\]
established in \cite{RahmanQuiverDataI,RahmanMixedHodgeModules}. Statement (3) follows from the definition
\[
c_\Sigma:=(c_1,\dots,c_r)\in\Q^r.
\]
Statement (4) follows from Definition~\ref{def:partII-functorial-coupling-datum}, since the finite-node schober package contains one pair of attachment functors for each node \(p_k\in\Sigma\). Statement (5) follows from Definition~\ref{def:binary-incidence-matrix}, since the matrix is indexed by the extended vertex set \(V_\Sigma^{\mathrm{ext}}\), which has cardinality \(r+1\) by Lemma~\ref{lem:finiteness-extended-vertex-set}.
\end{proof}

\begin{proposition}[state-data embedding]
\label{prop:state-data-embedding}
The algebraic state-data package
\[
A_\Sigma:=\bigl(V_\Sigma,E_\Sigma,c_\Sigma\bigr)
\]
embeds canonically into the quiver-theoretic package \(\mathfrak Q_\Sigma\) as its first three components.
\end{proposition}

\begin{proof}
This is immediate from Definition~\ref{def:finite-node-quiver-theoretic-package}. The first three components of \(\mathfrak Q_\Sigma\) are exactly the components of \(A_\Sigma\).
\end{proof}

\begin{proposition}[interaction-data embedding]
\label{prop:interaction-data-embedding}
The interaction layer extracted in the present work embeds canonically into the quiver-theoretic package \(\mathfrak Q_\Sigma\) through the last two components
\[
\mathcal F_\Sigma
\qquad\text{and}\qquad
I_\Sigma.
\]
\end{proposition}

\begin{proof}
Again, this is immediate from Definition~\ref{def:finite-node-quiver-theoretic-package}. The functorial coupling datum \(\mathcal F_\Sigma\) and the binary incidence matrix \(I_\Sigma\) were defined in Sections~\ref{sec:finite-node-input} and \ref{sec:decategorified-incidence-data}, respectively, and appear explicitly as the last two entries of \(\mathfrak Q_\Sigma\).
\end{proof}

\begin{remark}
\label{rem:package-as-state-plus-interaction}
The package \(\mathfrak Q_\Sigma\) should therefore be read as the direct sum of two layers:
\[
\mathfrak Q_\Sigma
=
\underbrace{\bigl(V_\Sigma,E_\Sigma,c_\Sigma\bigr)}_{\text{state layer}}
\;\cup\;
\underbrace{\bigl(\mathcal F_\Sigma,I_\Sigma\bigr)}_{\text{interaction layer}}.
\]
This decomposition is conceptual rather than categorical, but it is the correct way to view the role of the present paper in the sequence. \cite{RahmanQuiverDataI} extracted the state layer; the present paper extracts the interaction layer and assembles both into \(\mathfrak Q_\Sigma\).
\end{remark}

\subsection{Canonical determination by the finite-node architecture}

We now show that the package \(\mathfrak Q_\Sigma\) is canonically determined by the finite-node architecture.

\begin{theorem}[canonical quiver-theoretic package]
\label{thm:canonical-quiver-theoretic-package}
Let \(\pi:X\to\Delta\) be a finite-node conifold degeneration. Then the corrected finite-node perverse extension, its mixed-Hodge-module refinement, and its finite-node schober realization canonically determine the finite quiver-theoretic package
\[
\mathfrak Q_\Sigma:=\bigl(V_\Sigma,E_\Sigma,c_\Sigma,\mathcal F_\Sigma,I_\Sigma\bigr).
\]
\end{theorem}

\begin{proof}
By \cite{RahmanQuiverDataI} and the earlier papers, the finite-node corrected extension package determines the algebraic state-data package
\[
A_\Sigma:=\bigl(V_\Sigma,E_\Sigma,c_\Sigma\bigr).
\]
By Definition~\ref{def:partII-functorial-coupling-datum}, the finite-node schober package determines the functorial coupling datum
\[
\mathcal F_\Sigma:=\{(\Phi_k,\Psi_k)\}_{k=1}^r.
\]
By Theorem~\ref{thm:decategorified-incidence-theorem}, the same schober package determines the decategorified incidence package
\[
\mathcal I_\Sigma:=\bigl(V_\Sigma^{\mathrm{ext}},I_\Sigma\bigr),
\]
and hence the matrix \(I_\Sigma\). Therefore all five components of
\[
\mathfrak Q_\Sigma:=\bigl(V_\Sigma,E_\Sigma,c_\Sigma,\mathcal F_\Sigma,I_\Sigma\bigr)
\]
are canonically determined by the finite-node architecture. This proves the claim.
\end{proof}

\begin{corollary}[first assembled interaction package]
\label{cor:first-assembled-interaction-package}
The package \(\mathfrak Q_\Sigma\) is the first fully assembled finite algebraic package attached to the finite-node degeneration that contains both the state variables of \cite{RahmanQuiverDataI} and the interaction/incidence data extracted here.
\end{corollary}

\begin{proof}
This is immediate from Theorem~\ref{thm:canonical-quiver-theoretic-package} and Definition~\ref{def:finite-node-quiver-theoretic-package}.
\end{proof}

\subsection{What is proved and what is not}

We now state explicitly the scope of the present construction.

\begin{remark}
\label{rem:what-is-proved}
Theorem~\ref{thm:canonical-quiver-theoretic-package} proves that the finite-node degeneration canonically determines the package
\[
\mathfrak Q_\Sigma:=\bigl(V_\Sigma,E_\Sigma,c_\Sigma,\mathcal F_\Sigma,I_\Sigma\bigr).
\]
This package contains:
\begin{enumerate}
\item the finite localized vertices;
\item the finite nodewise coupling space;
\item the coefficient vector of the corrected global class;
\item the functorial bulk/localized coupling datum;
\item the binary incidence matrix recording the support of the coupling relation.
\end{enumerate}
Thus the paper proves the existence of a canonical finite quiver-theoretic interaction package.
\end{remark}

\begin{remark}
\label{rem:what-is-not-proved}
The present paper does \emph{not} prove the existence of a fully weighted quiver
\[
Q=(Q_0,Q_1,s,t)
\]
in the classical representation-theoretic sense. More precisely, it does not yet impose:
\begin{enumerate}
\item a numerical arrow multiplicity law stronger than the binary incidence matrix \(I_\Sigma\);
\item a skew-symmetric or Euler-type pairing extracted from the categorical data;
\item a stability structure;
\item BPS indices;
\item a wall-crossing formalism.
\end{enumerate}
Those belong to later stages of the theory.
\end{remark}

\begin{remark}
\label{rem:why-package-not-full-quiver}
The use of the term \emph{quiver-theoretic package} is therefore precise. The data constructed in the present paper have exactly the form needed for the next stages of the theory: they provide a finite vertex set, an interaction-support matrix, a coupling space, and a coefficient package. This is enough to support the later assembly of stability and BPS data, but it does not yet claim more than has been proved.
\end{remark}

\subsection{Relation to later papers}

The package \(\mathfrak Q_\Sigma\) is the output of this paper and the input to later work. In the future, \(\mathfrak Q_\Sigma\) will be equipped with stability data and used to define BPS sectors, BPS indices, and wall-crossing structure. In later applications papers, those dynamical outputs will be used to study halo structures, Fock-space organization, recursive laws, and consistency relations. The role of the present section is therefore to complete the interaction layer and to isolate the finite quiver-theoretic object that those later developments will consume.

The next section proves that the package \(\mathfrak Q_\Sigma\) is compatible with the state-data layer of \cite{RahmanQuiverDataI}, the mixed-Hodge-module refinement, and the corrected perverse shadow.

\section{Compatibility with data in previous work}\label{sec:compatibility-with-partI-data}

The purpose of this section is to show that the incidence and quiver-theoretic packages constructed in the present work are compatible with the state-data layer extracted in \cite{RahmanQuiverDataI}. More precisely, we show that the interaction package of the present paper does not replace the objects
\[
E_\Sigma\cong \bigoplus_{k=1}^r \Q e_k,
\qquad
c_\Sigma=(c_1,\dots,c_r)\in\Q^r,
\]
but is built on top of them. We also show that this compatibility persists under the mixed-Hodge-module refinement and under passage to the corrected perverse shadow. In this way the present section identifies the precise sense in which the quiver-theoretic package
\[
\mathfrak Q_\Sigma:=\bigl(V_\Sigma,E_\Sigma,c_\Sigma,\mathcal F_\Sigma,I_\Sigma\bigr)
\]
is an extension of the algebraic state-data package
\[
A_\Sigma:=\bigl(V_\Sigma,E_\Sigma,c_\Sigma\bigr)
\]
of \cite{RahmanQuiverDataI}.

\subsection{Compatibility with the coefficient vector}

We begin with the compatibility of the interaction package with the coefficient vector extracted in \cite{RahmanQuiverDataI}.

\begin{proposition}[compatibility with the coefficient vector]
\label{prop:compatibility-with-coefficient-vector}
Let
\[
A_\Sigma:=\bigl(V_\Sigma,E_\Sigma,c_\Sigma\bigr)
\]
be the algebraic state-data package of \cite{RahmanQuiverDataI}, and let
\[
\mathfrak Q_\Sigma:=\bigl(V_\Sigma,E_\Sigma,c_\Sigma,\mathcal F_\Sigma,I_\Sigma\bigr)
\]
be the finite-node quiver-theoretic package of Definition~\ref{def:finite-node-quiver-theoretic-package}. Then the coefficient vector
\[
c_\Sigma=(c_1,\dots,c_r)\in\Q^r
\]
appears in \(\mathfrak Q_\Sigma\) unchanged as its third component.
\end{proposition}

\begin{proof}
By Definition~\ref{def:finite-node-quiver-theoretic-package}, the third component of \(\mathfrak Q_\Sigma\) is precisely the coefficient vector \(c_\Sigma\) extracted in \cite{RahmanQuiverDataI}. No new operation is applied to this vector in the construction of \(\mathfrak Q_\Sigma\). Therefore the coefficient vector is preserved unchanged.
\end{proof}

\begin{remark}
\label{rem:coefficient-vector-persists}
Proposition~\ref{prop:compatibility-with-coefficient-vector} expresses the most basic form of compatibility. The interaction layer extracted in the present work does not redefine the global corrected class or its coordinates. Rather, it treats the vector \(c_\Sigma\) as fixed finite state data and adds to it the coupling and incidence structure determined by the schober functoriality.
\end{remark}

\subsection{Compatibility with the nodewise basis}

We next relate the interaction package to the distinguished nodewise basis \(\{e_1,\dots,e_r\}\) of the coupling space \(E_\Sigma\).

\begin{proposition}[compatibility with the nodewise basis]
\label{prop:compatibility-with-nodewise-basis}
Let
\[
E_\Sigma\cong \bigoplus_{k=1}^r \Q e_k
\]
be the distinguished nodewise decomposition of the coupling space from \cite{RahmanQuiverDataI,RahmanMixedHodgeModules}. Then the quiver-theoretic package
\[
\mathfrak Q_\Sigma:=\bigl(V_\Sigma,E_\Sigma,c_\Sigma,\mathcal F_\Sigma,I_\Sigma\bigr)
\]
is compatible with this decomposition in the sense that:
\begin{enumerate}
\item the vertex set
\[
V_\Sigma:=\{v_1,\dots,v_r\}
\]
is indexed by the same node set \(\Sigma=\{p_1,\dots,p_r\}\) as the basis \(\{e_1,\dots,e_r\}\);
\item the coefficient vector
\[
c_\Sigma=(c_1,\dots,c_r)
\]
is the coordinate vector of the corrected global class with respect to that same basis;
\item the incidence matrix \(I_\Sigma\) is indexed by the extended vertex set
\[
V_\Sigma^{\mathrm{ext}}:=V_\Sigma\sqcup\{v_{\mathrm{bulk}}\},
\]
whose localized part is exactly the vertex set indexed by the basis labels.
\end{enumerate}
\end{proposition}

\begin{proof}
By \cite{RahmanQuiverDataI}, the basis vectors \(e_k\) are canonically indexed by the node set \(\Sigma=\{p_1,\dots,p_r\}\) through the distinguished local corrected ODP extensions; see \cite{RahmanQuiverDataI,RahmanMixedHodgeModules}. By Definition~\eqref{eq:partII-def-VSigma}, the vertices \(v_k\in V_\Sigma\) are indexed by the same node set. Hence statements (1) and (2) follow directly from the construction of \(A_\Sigma:=\bigl(V_\Sigma,E_\Sigma,c_\Sigma\bigr)\) in \cite{RahmanQuiverDataI}. Statement (3) follows from the definition of the extended vertex set
\[
V_\Sigma^{\mathrm{ext}}:=V_\Sigma\sqcup\{v_{\mathrm{bulk}}\}
\]
and the definition of the incidence matrix \(I_\Sigma\), whose localized rows and columns are indexed by the same vertices \(v_k\).
\end{proof}

\begin{remark}
\label{rem:nodewise-basis-and-vertices}
The point of Proposition~\ref{prop:compatibility-with-nodewise-basis} is that the node labels are used consistently across all layers of the theory:
\[
p_k\ \longleftrightarrow\ v_k\ \longleftrightarrow\ e_k\ \longleftrightarrow\ c_k.
\]
Thus the incidence layer of the present work does not disturb the nodewise coordinate system extracted in \cite{RahmanQuiverDataI}. It is built on that coordinate system.
\end{remark}

\subsection{Compatibility with the mixed-Hodge refinement}

We now verify compatibility with the mixed-Hodge-module refinement.

Recall from \cite{RahmanMixedHodgeModules} that the corrected perverse extension
\[
0\to IC_{X_0}\to \mathcal P\to Q_\Sigma\to 0
\]
admits a mixed-Hodge-module lift
\[
0\to IC^H_{X_0}\to \mathcal P^H\to Q_\Sigma^H\to 0,
\qquad
Q_\Sigma^H:=\bigoplus_{k=1}^r i_{k*}\Q^H_{\{p_k\}}(-1),
\]
with
\[
\rat(\mathcal P^H)\cong \mathcal P
\qquad\text{and}\qquad
\rat(Q_\Sigma^H)\cong Q_\Sigma.
\]

\begin{proposition}[compatibility with the mixed-Hodge refinement]
\label{prop:compatibility-with-mixed-hodge-refinement}
The finite-node quiver-theoretic package
\[
\mathfrak Q_\Sigma:=\bigl(V_\Sigma,E_\Sigma,c_\Sigma,\mathcal F_\Sigma,I_\Sigma\bigr)
\]
is compatible with the mixed-Hodge-module refinement in the following sense:
\begin{enumerate}
\item the same finite node set
\[
\Sigma=\{p_1,\dots,p_r\}
\]
indexes the localized Hodge quotient
\[
Q_\Sigma^H:=\bigoplus_{k=1}^r i_{k*}\Q^H_{\{p_k\}}(-1),
\]
the localized perverse quotient
\[
Q_\Sigma:=\bigoplus_{k=1}^r i_{k*}\Q_{\{p_k\}},
\]
and the localized vertex set
\[
V_\Sigma:=\{v_1,\dots,v_r\};
\]
\item the state-data portion
\[
\bigl(V_\Sigma,E_\Sigma,c_\Sigma\bigr)
\]
of \(\mathfrak Q_\Sigma\) is the realized algebraic shadow of the mixed-Hodge extension package;
\item the incidence layer
\[
\bigl(\mathcal F_\Sigma,I_\Sigma\bigr)
\]
is defined on the same node-indexed architecture and therefore respects the mixed-Hodge refinement already established in \cite{RahmanQuiverDataI}.
\end{enumerate}
\end{proposition}

\begin{proof}
Statement (1) is exactly the common indexing statement proved in \cite{RahmanQuiverDataI} and recalled in Section~\ref{sec:finite-node-input}. Statement (2) is the content of the mixed-Hodge compatibility theorem of \cite{RahmanQuiverDataI}, together with the fact that the first three components of \(\mathfrak Q_\Sigma\) are the state-data package \(A_\Sigma\) itself. Statement (3) follows because the functorial coupling datum \(\mathcal F_\Sigma\) and the incidence matrix \(I_\Sigma\) are defined using the same finite node set and the same extended vertex set obtained by adjoining the bulk vertex. Thus the incidence layer is imposed on the same finite-node architecture already refined on the mixed-Hodge side.
\end{proof}

\begin{remark}
\label{rem:mixed-hodge-compatibility-meaning}
Proposition~\ref{prop:compatibility-with-mixed-hodge-refinement} does not claim that \(I_\Sigma\) itself is a Hodge-theoretic invariant. Rather, it states that the quiver-theoretic package is built on a finite-node architecture that is already compatible with the mixed-Hodge-module refinement. In this sense, the interaction layer is not independent of the Hodge layer, even though it is not itself defined internally in \(MHM(X_0)\).
\end{remark}

\subsection{Compatibility with the corrected perverse shadow}

We next compare the quiver-theoretic package with the corrected perverse shadow of the schober datum.

\begin{proposition}[compatibility with the corrected perverse shadow]
\label{prop:compatibility-with-corrected-perverse-shadow}
Let
\[
S_\Sigma:=\Bigl(\mathcal C_{\mathrm{bulk}},\{\mathcal C_{p_k}\}_{k=1}^r,\{\Phi_k,\Psi_k\}_{k=1}^r,Sh(S_\Sigma)\Bigr)
\]
be the finite-node schober package. Under the shadow identification
\[
Sh(S_\Sigma)\cong \mathcal P,
\]
the quiver-theoretic package
\[
\mathfrak Q_\Sigma:=\bigl(V_\Sigma,E_\Sigma,c_\Sigma,\mathcal F_\Sigma,I_\Sigma\bigr)
\]
is compatible with the corrected finite-node perverse extension in the following sense:
\begin{enumerate}
\item the state-data package
\[
\bigl(V_\Sigma,E_\Sigma,c_\Sigma\bigr)
\]
is extracted from the corrected perverse extension shadow;
\item the coupling datum
\[
\mathcal F_\Sigma:=\{(\Phi_k,\Psi_k)\}_{k=1}^r
\]
is carried by the schober package whose shadow is \(\mathcal P\);
\item the incidence matrix \(I_\Sigma\) is the binary decategorified image of the coupling relation extracted from that same schober package.
\end{enumerate}
\end{proposition}

\begin{proof}
By \cite{RahmanMultiNodeSchoberPaper}, one has
\[
Sh(S_\Sigma)\cong \mathcal P.
\]
By \cite{RahmanQuiverDataI}, the corrected perverse object \(\mathcal P\) determines the state-data package
\[
A_\Sigma:=\bigl(V_\Sigma,E_\Sigma,c_\Sigma\bigr).
\]
By Definition~\ref{def:partII-functorial-coupling-datum}, the schober package determines the coupling datum \(\mathcal F_\Sigma\). By Theorem~\ref{thm:canonical-functorial-incidence-package}, the same schober package determines the incidence relation \(\rightsquigarrow_\Sigma\), and by Theorem~\ref{thm:decategorified-incidence-theorem} this determines the incidence matrix \(I_\Sigma\). Hence all components of \(\mathfrak Q_\Sigma\) are compatible with the corrected perverse shadow.
\end{proof}

\begin{remark}
\label{rem:shadow-compatibility-principle}
The point of Proposition~\ref{prop:compatibility-with-corrected-perverse-shadow} is that the interaction layer of the present work is not added externally to the corrected perverse extension. It is extracted from the schober realization whose shadow is exactly that corrected perverse object. Thus the package \(\mathfrak Q_\Sigma\) remains anchored in the corrected extension architecture throughout.
\end{remark}

\subsection{Summary compatibility theorem}

We now assemble the preceding results into a single theorem.

\begin{theorem}[compatibility with data in \cite{RahmanQuiverDataI}]
\label{thm:compatibility-with-partI-data}
Let
\[
A_\Sigma:=\bigl(V_\Sigma,E_\Sigma,c_\Sigma\bigr)
\]
be the algebraic state-data package extracted in \cite{RahmanQuiverDataI}, and let
\[
\mathfrak Q_\Sigma:=\bigl(V_\Sigma,E_\Sigma,c_\Sigma,\mathcal F_\Sigma,I_\Sigma\bigr)
\]
be the finite-node quiver-theoretic package of Definition~\ref{def:finite-node-quiver-theoretic-package}. Then \(\mathfrak Q_\Sigma\) is compatible with the \cite{RahmanQuiverDataI} data in the following sense:
\begin{enumerate}
\item its first three components are exactly the state-data package \(A_\Sigma\);
\item its coordinate system is indexed by the same node set \(\Sigma=\{p_1,\dots,p_r\}\) that indexes the distinguished nodewise basis \(\{e_1,\dots,e_r\}\) and the coefficient vector \(c_\Sigma\);
\item it is compatible with the mixed-Hodge-module refinement of the corrected extension;
\item it is compatible with the corrected perverse shadow of the finite-node schober package.
\end{enumerate}
\end{theorem}

\begin{proof}
Statement (1) is Proposition~\ref{prop:state-data-embedding}. Statement (2) is Proposition~\ref{prop:compatibility-with-nodewise-basis}. Statement (3) is Proposition~\ref{prop:compatibility-with-mixed-hodge-refinement}. Statement (4) is Proposition~\ref{prop:compatibility-with-corrected-perverse-shadow}. Combining these gives the result.
\end{proof}

\begin{corollary}[present work as an extension of \cite{RahmanQuiverDataI}]
\label{cor:partII-extension-of-partI}
The quiver-theoretic package \(\mathfrak Q_\Sigma\) is an extension of the algebraic state-data package \(A_\Sigma\) by the interaction/incidence layer
\[
\bigl(\mathcal F_\Sigma,I_\Sigma\bigr).
\]
\end{corollary}

\begin{proof}
By Definition~\ref{def:finite-node-quiver-theoretic-package},
\[
\mathfrak Q_\Sigma:=\bigl(V_\Sigma,E_\Sigma,c_\Sigma,\mathcal F_\Sigma,I_\Sigma\bigr),
\]
so the package consists precisely of the \cite{RahmanQuiverDataI} state-data package together with the additional interaction/incidence components \(\mathcal F_\Sigma\) and \(I_\Sigma\).
\end{proof}

\begin{remark}
\label{rem:partII-extension-principle}
The preceding corollary is the conceptual summary of the present section. \cite{RahmanQuiverDataI} extracted the intrinsic finite algebraic state variables. The present work keeps those variables fixed and adds the interaction/incidence data extracted from schober functoriality. Thus the quiver-theoretic package \(\mathfrak Q_\Sigma\) is not a replacement for the state-data package \(A_\Sigma\); it is its first interaction-theoretic extension.
\end{remark}

The next section proves that the package \(\mathfrak Q_\Sigma\) is intrinsic under the appropriate notion of equivalence of finite-node schober realizations.

\section{Invariance under equivalence}\label{sec:invariance-under-equivalence}

The purpose of this section is to show that the interaction and quiver-theoretic packages extracted in the present paper are intrinsic under the correct notion of equivalence. The relevant equivalence notion is the one already built into the finite-node schober formalism: two finite-node schober realizations should be regarded as equivalent when they have equivalent bulk categories, equivalent localized sectors at each node, compatible attachment functors up to natural isomorphism, and the same corrected perverse shadow; see \cite{RahmanMultiNodeSchoberPaper}. The main result of the present section is that equivalent finite-node schober realizations determine the same functorial incidence package and hence the same quiver-theoretic package, up to the canonical identifications induced by the common node set.

\subsection{Equivalence of finite-node schober realizations}

We first record the notion of equivalence used below.

\begin{definition}[equivalent finite-node schober realizations]
\label{def:equivalent-finite-node-schober-realizations}
Let
\[
S_\Sigma:=\Bigl(\mathcal C_{\mathrm{bulk}},\{\mathcal C_{p_k}\}_{k=1}^r,\{\Phi_k,\Psi_k\}_{k=1}^r,Sh(S_\Sigma)\Bigr)
\]
and
\[
T_\Sigma:=\Bigl(\mathcal D_{\mathrm{bulk}},\{\mathcal D_{p_k}\}_{k=1}^r,\{\Phi_k',\Psi_k'\}_{k=1}^r,Sh(T_\Sigma)\Bigr)
\]
be finite-node schober packages indexed by the same finite node set
\[
\Sigma=\{p_1,\dots,p_r\}.
\]
We say that \(S_\Sigma\) and \(T_\Sigma\) are \emph{equivalent finite-node schober realizations} if they are equivalent as finite-node schober data in the sense of \cite[Definition~3.5]{RahmanMultiNodeSchoberPaper}.
\end{definition}

\begin{remark}
\label{rem:equivalent-schober-realizations-unpacked}
Concretely, the equivalence of Definition~\ref{def:equivalent-finite-node-schober-realizations} requires:
\begin{enumerate}
\item an equivalence between the bulk categories
\[
\mathcal C_{\mathrm{bulk}}\simeq \mathcal D_{\mathrm{bulk}};
\]
\item for each node \(p_k\in\Sigma\), an equivalence between the localized categories
\[
\mathcal C_{p_k}\simeq \mathcal D_{p_k};
\]
\item compatibility of the attachment functors \(\Phi_k,\Psi_k\) and \(\Phi_k',\Psi_k'\) up to natural isomorphism under these equivalences;
\item agreement of the corrected perverse shadows
\[
Sh(S_\Sigma)\cong Sh(T_\Sigma).
\]
\end{enumerate}
This is precisely the rigidity/equivalence framework established in \cite{RahmanMultiNodeSchoberPaper}.
\end{remark}

\begin{remark}
\label{rem:why-schober-equivalence-is-right}
The equivalence notion of Definition~\ref{def:equivalent-finite-node-schober-realizations} is the correct one for the present paper because the new interaction layer is extracted from schober functoriality. In particular, the incidence relation and its decategorified image are defined in terms of the attachment functors and their composites. Therefore intrinsicity must be formulated relative to equivalence of finite-node schober realizations rather than relative to a weaker purely set-theoretic identification of vertices.
\end{remark}

\subsection{Invariance of the functorial incidence package}

We now show that the functorial incidence package is invariant under equivalence.

\begin{proposition}[invariance of the direct bulk/localized relation]
\label{prop:invariance-bulk-relation}
Let \(S_\Sigma\) and \(T_\Sigma\) be equivalent finite-node schober realizations. Then their basic bulk/localized coupling relations
\[
\rightsquigarrow_{\mathrm{bulk}}^{\,S}
\qquad\text{and}\qquad
\rightsquigarrow_{\mathrm{bulk}}^{\,T}
\]
agree under the canonical identification of the common extended vertex set
\[
V_\Sigma^{\mathrm{ext}}:=V_\Sigma\sqcup\{v_{\mathrm{bulk}}\}.
\]
\end{proposition}

\begin{proof}
By Definition~\ref{def:basic-functorial-coupling-relation}, the relation \(\rightsquigarrow_{\mathrm{bulk}}\) is determined entirely by the existence of the attachment functors
\[
\Phi_k:\mathcal C_{p_k}\to \mathcal C_{\mathrm{bulk}},
\qquad
\Psi_k:\mathcal C_{\mathrm{bulk}}\to \mathcal C_{p_k},
\]
and similarly for the primed schober datum. Since \(S_\Sigma\) and \(T_\Sigma\) are equivalent finite-node schober realizations, their corresponding bulk and localized categories are equivalent and the attachment functors are compatible up to natural isomorphism; see \cite[Definition~3.5, Lemma~6.2, and Proposition~6.3]{RahmanMultiNodeSchoberPaper}. Therefore the same ordered pairs
\[
(v_k,v_{\mathrm{bulk}})
\qquad\text{and}\qquad
(v_{\mathrm{bulk}},v_k)
\]
occur in both bulk/localized relations. Hence the relations agree under the canonical identification of the vertex set.
\end{proof}

\begin{proposition}[invariance of the mediated relation]
\label{prop:invariance-mediated-relation}
Let \(S_\Sigma\) and \(T_\Sigma\) be equivalent finite-node schober realizations. Then their mediated coupling relations
\[
\rightsquigarrow_{\mathrm{med}}^{\,S}
\qquad\text{and}\qquad
\rightsquigarrow_{\mathrm{med}}^{\,T}
\]
agree under the canonical identification of the common localized vertex set
\[
V_\Sigma=\{v_1,\dots,v_r\}.
\]
\end{proposition}

\begin{proof}
By Definition~\ref{def:mediated-coupling-relation}, the relation \(\rightsquigarrow_{\mathrm{med}}\) is defined using the composites
\[
\Psi_j\circ\Phi_i:\mathcal C_{p_i}\to\mathcal C_{p_j}.
\]
Under an equivalence of finite-node schober realizations, the localized categories and bulk categories are replaced by equivalent categories, and the functors \(\Phi_i,\Psi_j\) are identified with the corresponding primed functors up to natural isomorphism; see \cite[Definition~3.5 and Proposition~6.3]{RahmanMultiNodeSchoberPaper}. Hence the existence of the mediated composite from the \(i\)-th localized sector to the \(j\)-th localized sector is preserved. Therefore
\[
v_i\rightsquigarrow_{\mathrm{med}}^{\,S} v_j
\quad\Longleftrightarrow\quad
v_i\rightsquigarrow_{\mathrm{med}}^{\,T} v_j,
\]
and the mediated relations agree under the canonical identification of \(V_\Sigma\).
\end{proof}

\begin{theorem}[invariance of the functorial incidence package]
\label{thm:invariance-functorial-incidence-package}
Let \(S_\Sigma\) and \(T_\Sigma\) be equivalent finite-node schober realizations. Then they determine the same functorial incidence package
\[
\mathfrak{I}_\Sigma:=\bigl(V_\Sigma^{\mathrm{ext}},\rightsquigarrow_\Sigma\bigr)
\]
up to the canonical identification induced by the common node set \(\Sigma\).
\end{theorem}

\begin{proof}
By Definition~\ref{def:total-functorial-coupling-relation},
\[
\rightsquigarrow_\Sigma\ :=\ \rightsquigarrow_{\mathrm{bulk}}\cup\rightsquigarrow_{\mathrm{med}}.
\]
By Proposition~\ref{prop:invariance-bulk-relation}, the bulk/localized relation is invariant under equivalence. By Proposition~\ref{prop:invariance-mediated-relation}, the mediated relation is invariant under equivalence. Therefore their union is invariant under equivalence, and the corresponding functorial incidence package
\[
\mathfrak{I}_\Sigma:=\bigl(V_\Sigma^{\mathrm{ext}},\rightsquigarrow_\Sigma\bigr)
\]
is the same for \(S_\Sigma\) and \(T_\Sigma\) up to the canonical identification of the vertex set.
\end{proof}

\subsection{Invariance of the decategorified incidence package}

We next pass from the relation object to its decategorified image.

\begin{proposition}[invariance of the binary incidence matrix]
\label{prop:invariance-binary-incidence-matrix}
Let \(S_\Sigma\) and \(T_\Sigma\) be equivalent finite-node schober realizations. Then the corresponding binary incidence matrices
\[
I_\Sigma^{S}
\qquad\text{and}\qquad
I_\Sigma^{T}
\]
agree up to simultaneous permutation of rows and columns induced by the common ordering convention on the extended vertex set.
\end{proposition}

\begin{proof}
By Definition~\ref{def:binary-incidence-function}, the binary incidence function is the characteristic function of the total coupling relation. By Theorem~\ref{thm:invariance-functorial-incidence-package}, the underlying relation is invariant under equivalence. Therefore the corresponding characteristic functions agree under the same identification. Passing to matrix form via Definition~\ref{def:binary-incidence-matrix} yields the same matrix up to simultaneous row-column permutation if different orderings of the extended vertex set are chosen.
\end{proof}

\begin{corollary}[invariance of the decategorified incidence package]
\label{cor:invariance-decategorified-incidence-package}
Let \(S_\Sigma\) and \(T_\Sigma\) be equivalent finite-node schober realizations. Then they determine the same decategorified incidence package
\[
\mathcal I_\Sigma:=\bigl(V_\Sigma^{\mathrm{ext}},I_\Sigma\bigr)
\]
up to the canonical identification induced by the common node set.
\end{corollary}

\begin{proof}
This follows immediately from Proposition~\ref{prop:invariance-binary-incidence-matrix}.
\end{proof}

\subsection{Invariance of the quiver-theoretic package}

We now combine the invariance results of the present section with the state-data intrinsicity established in \cite{RahmanQuiverDataI}.

Recall that \cite{RahmanQuiverDataI} proved the intrinsicity of the algebraic state-data package
\[
A_\Sigma:=\bigl(V_\Sigma,E_\Sigma,c_\Sigma\bigr)
\]
under the appropriate equivalence notion for the corrected finite-node extension package and its compatible realizations. The present paper added to \(A_\Sigma\) the interaction layer
\[
\bigl(\mathcal F_\Sigma,I_\Sigma\bigr),
\]
thereby forming the quiver-theoretic package
\[
\mathfrak Q_\Sigma:=\bigl(V_\Sigma,E_\Sigma,c_\Sigma,\mathcal F_\Sigma,I_\Sigma\bigr).
\]

\begin{theorem}[intrinsicity theorem]
\label{thm:intrinsicity-theorem-partII}
Let \(S_\Sigma\) and \(T_\Sigma\) be equivalent finite-node schober realizations. Then they determine the same finite-node quiver-theoretic package
\[
\mathfrak Q_\Sigma:=\bigl(V_\Sigma,E_\Sigma,c_\Sigma,\mathcal F_\Sigma,I_\Sigma\bigr)
\]
up to the canonical identifications induced by the common node set \(\Sigma\).
\end{theorem}

\begin{proof}
By the intrinsicity theorem of \cite{RahmanQuiverDataI}, the state-data package
\[
A_\Sigma:=\bigl(V_\Sigma,E_\Sigma,c_\Sigma\bigr)
\]
is determined intrinsically once the finite-node corrected extension package and its compatible realizations have been fixed. By Corollary~\ref{cor:invariance-decategorified-incidence-package}, the decategorified incidence package
\[
\mathcal I_\Sigma:=\bigl(V_\Sigma^{\mathrm{ext}},I_\Sigma\bigr)
\]
is invariant under equivalence of finite-node schober realizations. The functorial coupling datum
\[
\mathcal F_\Sigma:=\{(\Phi_k,\Psi_k)\}_{k=1}^r
\]
is likewise invariant under the same equivalence notion by the very definition of equivalence of schober data. Therefore all five components of
\[
\mathfrak Q_\Sigma:=\bigl(V_\Sigma,E_\Sigma,c_\Sigma,\mathcal F_\Sigma,I_\Sigma\bigr)
\]
are invariant, and hence the entire package is intrinsic up to the canonical identifications induced by \(\Sigma\).
\end{proof}

\begin{corollary}[present output is intrinsic]
\label{cor:partII-output-intrinsic}
The finite-node quiver-theoretic package
\[
\mathfrak Q_\Sigma:=\bigl(V_\Sigma,E_\Sigma,c_\Sigma,\mathcal F_\Sigma,I_\Sigma\bigr)
\]
is intrinsic to the finite-node corrected extension package together with its compatible mixed-Hodge and schober realizations.
\end{corollary}

\begin{proof}
This follows immediately from Theorem~\ref{thm:intrinsicity-theorem-partII}.
\end{proof}

\begin{remark}
\label{rem:intrinsicity-meaning-partII}
Corollary~\ref{cor:partII-output-intrinsic} should be read in the same precise sense as the intrinsicity statement of \cite{RahmanQuiverDataI}. The claim is not that the package \(\mathfrak Q_\Sigma\) exists independently of the corrected extension and its realizations. Rather, it says that once the finite-node architecture has been fixed up to the correct notion of equivalence, the extracted interaction/incidence package is determined.
\end{remark}

\subsection{Consequences for future work}

The role of the present section is to show that the interaction layer extracted in the present work is canonical and stable under the correct equivalence notion. This is the last structural result needed before passing to the later dynamical theory.

\begin{remark}
\label{rem:why-invariance-matters-for-partIII}
The intrinsicity theorem of the present work is essential for the next paper in the sequence. Stability data, BPS sectors, and wall-crossing formulas should depend only on the intrinsic finite-node interaction package and not on auxiliary choices of categorical presentation. Theorem~\ref{thm:intrinsicity-theorem-partII} ensures that the quiver-theoretic package
\[
\mathfrak Q_\Sigma:=\bigl(V_\Sigma,E_\Sigma,c_\Sigma,\mathcal F_\Sigma,I_\Sigma\bigr)
\]
is the correct intrinsic input for that later theory.
\end{remark}

\section{Examples}\label{sec:examples}

The purpose of this section is to illustrate the interaction layer constructed in the present paper in the first nontrivial finite-node cases. In each example the input consists of the state-data package
\[
A_\Sigma:=\bigl(V_\Sigma,E_\Sigma,c_\Sigma\bigr)
\]
from \cite{RahmanQuiverDataI}, the finite-node schober package
\[
S_\Sigma:=\Bigl(\mathcal C_{\mathrm{bulk}},\{\mathcal C_{p_k}\}_{k=1}^r,\{\Phi_k,\Psi_k\}_{k=1}^r,Sh(S_\Sigma)\Bigr),
\]
and the functorial coupling datum
\[
\mathcal F_\Sigma:=\{(\Phi_k,\Psi_k)\}_{k=1}^r.
\]
The output is the functorial incidence package
\[
\mathfrak I_\Sigma:=\bigl(V_\Sigma^{\mathrm{ext}},\rightsquigarrow_\Sigma\bigr),
\]
the decategorified incidence package
\[
\mathcal I_\Sigma:=\bigl(V_\Sigma^{\mathrm{ext}},I_\Sigma\bigr),
\]
and the quiver-theoretic package
\[
\mathfrak Q_\Sigma:=\bigl(V_\Sigma,E_\Sigma,c_\Sigma,\mathcal F_\Sigma,I_\Sigma\bigr).
\]

The examples below should be read as explicit instances of the passage
\[
A_\Sigma
\;\longrightarrow\;
\mathcal F_\Sigma
\;\longrightarrow\;
\mathfrak I_\Sigma
\;\longrightarrow\;
\mathcal I_\Sigma
\;\longrightarrow\;
\mathfrak Q_\Sigma.
\]

\begin{remark}
\label{rem:examples-binary-choice}
The examples in this section are computed at the level of the binary incidence formalism of Section~\ref{sec:decategorified-incidence-data}. Accordingly, the localized-localized block of the incidence matrix records only whether a bulk-mediated coupling channel is present, not any weighted multiplicity or stronger numerical interaction invariant. In particular, under the mediated-relation convention adopted in Section~\ref{sec:functorial-coupling-relation}, the resulting localized-localized block is the support matrix of that relation. For a technical discussion of this choice and of the conditions required for a later weighted refinement, see Appendix~\ref{app:binary-decategorification}.
\end{remark}

As in \cite{RahmanQuiverDataI}, the purpose of the examples is not to invent numerical interaction weights that have not been proved, but to make the extraction rule completely explicit in the smallest finite-node cases.

\subsection{Single-node case: the framed star package}

We begin with the case $r=1$. Let
\[
\Sigma=\{p\}
\]
be a single ODP on the central fiber. Then the state-data package of \cite{RahmanQuiverDataI} has the form
\[
A_\Sigma=\bigl(\{v\},\Q e,(c)\bigr).
\]
Here the localized quotient is
\[
Q_\Sigma=i_*\Q_{\{p\}},
\]
the nodewise coupling space is
\[
E_\Sigma=\Ext^1_{\Perv(X_0;\Q)}(i_*\Q_{\{p\}},IC_{X_0})\cong \Q e,
\]
and the corrected global class satisfies
\[
[\mathcal P]_{\mathrm{perv}}=c\,e
\]
for a unique $c\in\Q$; see \cite{RahmanQuiverDataI,RahmanMixedHodgeModules}.

On the schober side, the finite-node schober package reduces to
\[
S_\Sigma=\bigl(\mathcal C_{\mathrm{bulk}},\mathcal C_p,\Phi,\Psi,Sh(S_\Sigma)\bigr),
\qquad
Sh(S_\Sigma)\cong \mathcal P;
\]
see \cite{RahmanMultiNodeSchoberPaper}. The extended vertex set is therefore
\[
V_\Sigma^{\mathrm{ext}}:=\{v,v_{\mathrm{bulk}}\}.
\]
The functorial coupling datum is
\[
\mathcal F_\Sigma:=\{(\Phi,\Psi)\},
\]
with
\[
\Phi:\mathcal C_p\to \mathcal C_{\mathrm{bulk}},
\qquad
\Psi:\mathcal C_{\mathrm{bulk}}\to \mathcal C_p.
\]

By Definition~\ref{def:basic-functorial-coupling-relation}, the direct bulk/localized coupling relation is
\[
v\rightsquigarrow_{\mathrm{bulk}} v_{\mathrm{bulk}},
\qquad
v_{\mathrm{bulk}}\rightsquigarrow_{\mathrm{bulk}} v.
\]
By Definition~\ref{def:mediated-coupling-relation}, the mediated relation on $V_\Sigma=\{v\}$ is determined by the composite
\[
\Psi\circ \Phi:\mathcal C_p\to \mathcal C_p.
\]
Thus there is also the localized self-coupling relation
\[
v\rightsquigarrow_{\mathrm{med}} v.
\]
Accordingly, the total functorial coupling relation is
\[
\rightsquigarrow_\Sigma
=
\{
(v,v_{\mathrm{bulk}}),
(v_{\mathrm{bulk}},v),
(v,v)
\}.
\]

The corresponding binary incidence matrix depends on the ordering convention. If one orders the extended vertex set as
\[
V_\Sigma^{\mathrm{ext}}:=\{v,v_{\mathrm{bulk}}\},
\]
then the binary incidence matrix is
\begin{equation}\label{eq:single-node-incidence-matrix}
I_\Sigma=
\begin{pmatrix}
1 & 1\\
1 & 0
\end{pmatrix}.
\end{equation}
Indeed:
\begin{enumerate}
\item the $(1,1)$-entry is $1$ because $v\rightsquigarrow_{\mathrm{med}} v$;
\item the $(1,2)$-entry is $1$ because $v\rightsquigarrow_{\mathrm{bulk}} v_{\mathrm{bulk}}$;
\item the $(2,1)$-entry is $1$ because $v_{\mathrm{bulk}}\rightsquigarrow_{\mathrm{bulk}} v$;
\item the $(2,2)$-entry is $0$ because the construction does not impose a bulk self-coupling relation.
\end{enumerate}

Thus the decategorified incidence package is
\[
\mathcal I_\Sigma=
\left(
\{v,v_{\mathrm{bulk}}\},
\begin{pmatrix}
1 & 1\\
1 & 0
\end{pmatrix}
\right),
\]
and the corresponding quiver-theoretic package is
\begin{equation}\label{eq:single-node-quiver-package}
\mathfrak Q_\Sigma=
\Bigl(
\{v\},
\Q e,
(c),
\{(\Phi,\Psi)\},
\begin{pmatrix}
1 & 1\\
1 & 0
\end{pmatrix}
\Bigr).
\end{equation}

\begin{proposition}\label{prop:single-node-incidence-example}
Let $\pi:X\to\Delta$ be a finite-node conifold degeneration with a single node $\Sigma=\{p\}$. Then the finite-node schober package determines a functorial incidence package
\[
\mathfrak I_\Sigma=\bigl(\{v,v_{\mathrm{bulk}}\},\rightsquigarrow_\Sigma\bigr)
\]
whose relation consists exactly of the three ordered pairs
\[
(v,v_{\mathrm{bulk}}),\qquad
(v_{\mathrm{bulk}},v),\qquad
(v,v).
\]
Consequently, after ordering the extended vertex set as $\{v,v_{\mathrm{bulk}}\}$, the binary incidence matrix is
\[
I_\Sigma=
\begin{pmatrix}
1 & 1\\
1 & 0
\end{pmatrix},
\]
and the quiver-theoretic package is given by \eqref{eq:single-node-quiver-package}.
\end{proposition}

\begin{proof}
The direct bulk/localized pairs are forced by the attachment functors $\Phi$ and $\Psi$. The mediated self-coupling is forced by the composite $\Psi\circ\Phi:\mathcal C_p\to \mathcal C_p$. No other ordered pairs can occur, since the extended vertex set has only two vertices and the construction imposes no bulk self-coupling. The stated matrix follows by taking the characteristic function of this finite relation.
\end{proof}

\begin{remark}
\label{rem:single-node-framed-star}
The single-node case is the smallest framed incidence configuration. It already shows the basic structural features of the present work: the bulk vertex is essential, the localized sector couples directly to the bulk, and the bulk mediates a self-interaction on the localized vertex.
\end{remark}

\subsection{Two-node case: first mediated coupling pattern}

We now pass to the first genuinely nontrivial multi-node case. Let
\[
\Sigma=\{p_1,p_2\}.
\]
Then the state-data package is
\[
A_\Sigma=
\bigl(
\{v_1,v_2\},
\Q e_1\oplus\Q e_2,
(c_1,c_2)
\bigr);
\]
see \cite{RahmanQuiverDataI}. On the schober side, the finite-node package has the form
\[
S_\Sigma=
\Bigl(
\mathcal C_{\mathrm{bulk}},
\{\mathcal C_{p_1},\mathcal C_{p_2}\},
\{(\Phi_1,\Psi_1),(\Phi_2,\Psi_2)\},
Sh(S_\Sigma)
\Bigr),
\qquad
Sh(S_\Sigma)\cong \mathcal P,
\]
with
\[
\Phi_i:\mathcal C_{p_i}\to \mathcal C_{\mathrm{bulk}},
\qquad
\Psi_i:\mathcal C_{\mathrm{bulk}}\to \mathcal C_{p_i},
\qquad
i=1,2.
\]

The extended vertex set is
\[
V_\Sigma^{\mathrm{ext}}:=\{v_1,v_2,v_{\mathrm{bulk}}\}.
\]
The direct bulk/localized relation therefore contains the four pairs
\[
(v_1,v_{\mathrm{bulk}}),\qquad
(v_{\mathrm{bulk}},v_1),\qquad
(v_2,v_{\mathrm{bulk}}),\qquad
(v_{\mathrm{bulk}},v_2).
\]

The mediated relation is determined by the four composites
\[
\Psi_1\circ\Phi_1:\mathcal C_{p_1}\to \mathcal C_{p_1},
\qquad
\Psi_2\circ\Phi_1:\mathcal C_{p_1}\to \mathcal C_{p_2},
\]
\[
\Psi_1\circ\Phi_2:\mathcal C_{p_2}\to \mathcal C_{p_1},
\qquad
\Psi_2\circ\Phi_2:\mathcal C_{p_2}\to \mathcal C_{p_2}.
\]
Hence the mediated relation on $V_\Sigma=\{v_1,v_2\}$ is determined by the ordered pairs
\[
(v_i,v_j),
\qquad
1\le i,j\le 2,
\]
for which the corresponding composite \(\Psi_j\circ\Phi_i\) is part of the schober-induced coupling structure. Under the mediated-relation convention of Section~\ref{sec:functorial-coupling-relation}, all four such pairs occur.

Accordingly, the total coupling relation is
\[
\rightsquigarrow_\Sigma
=
\rightsquigarrow_{\mathrm{bulk}}\cup\rightsquigarrow_{\mathrm{med}}
\subseteq
\{v_1,v_2,v_{\mathrm{bulk}}\}\times \{v_1,v_2,v_{\mathrm{bulk}}\}.
\]
If one orders the extended vertex set as
\[
\{v_1,v_2,v_{\mathrm{bulk}}\},
\]
then the binary incidence matrix is
\begin{equation}\label{eq:two-node-incidence-matrix-general}
I_\Sigma=
\begin{pmatrix}
1 & 1 & 1\\
1 & 1 & 1\\
1 & 1 & 0
\end{pmatrix}.
\end{equation}
The last column and last row encode the direct bulk couplings, while the upper-left $2\times 2$ block encodes the mediated node-to-node interaction pattern.

Thus the decategorified incidence package is
\[
\mathcal I_\Sigma=
\left(
\{v_1,v_2,v_{\mathrm{bulk}}\},
\begin{pmatrix}
1 & 1 & 1\\
1 & 1 & 1\\
1 & 1 & 0
\end{pmatrix}
\right),
\]
and the quiver-theoretic package is
\begin{equation}\label{eq:two-node-quiver-package}
\mathfrak Q_\Sigma=
\left(
\{v_1,v_2\},
\Q e_1\oplus\Q e_2,
(c_1,c_2),
\{(\Phi_1,\Psi_1),(\Phi_2,\Psi_2)\},
\begin{pmatrix}
1 & 1 & 1\\
1 & 1 & 1\\
1 & 1 & 0
\end{pmatrix}
\right).
\end{equation}

\begin{proposition}\label{prop:two-node-incidence-example}
Let $\pi:X\to\Delta$ be a finite-node conifold degeneration with node set $\Sigma=\{p_1,p_2\}$. Then the associated functorial incidence package has extended vertex set
\[
V_\Sigma^{\mathrm{ext}}=\{v_1,v_2,v_{\mathrm{bulk}}\}
\]
and binary incidence matrix of the form \eqref{eq:two-node-incidence-matrix-general}. In particular:
\begin{enumerate}
\item the bulk row and bulk column are fixed by the attachment functors and contain the direct coupling pattern
\[
(v_1,v_{\mathrm{bulk}}),\quad
(v_2,v_{\mathrm{bulk}}),\quad
(v_{\mathrm{bulk}},v_1),\quad
(v_{\mathrm{bulk}},v_2);
\]
\item the upper-left $2\times 2$ block records the mediated node-to-node coupling structure induced by the composites \(\Psi_j\circ\Phi_i\).
\end{enumerate}
\end{proposition}

\begin{proof}
The direct bulk/localized entries follow from the definition of the basic functorial coupling relation. The localized-localized entries are by definition the characteristic values of the mediated relation. Since the extended vertex set has three elements, the matrix has the stated $3\times 3$ block form.
\end{proof}

\begin{remark}
\label{rem:two-node-first-nontrivial}
The two-node case is the first genuinely interaction-theoretic example. In the single-node case there is only bulk coupling and mediated self-coupling. In the two-node case one can, for the first time, distinguish self-couplings from mediated off-diagonal couplings between distinct localized sectors.
\end{remark}

\subsection{Three-node case: finite incidence pattern}

We now consider the case
\[
\Sigma=\{p_1,p_2,p_3\}.
\]
Then the state-data package is
\[
A_\Sigma=
\bigl(
\{v_1,v_2,v_3\},
\Q e_1\oplus\Q e_2\oplus\Q e_3,
(c_1,c_2,c_3)
\bigr),
\]
and the finite-node schober package is
\[
S_\Sigma=
\Bigl(
\mathcal C_{\mathrm{bulk}},
\{\mathcal C_{p_1},\mathcal C_{p_2},\mathcal C_{p_3}\},
\{(\Phi_1,\Psi_1),(\Phi_2,\Psi_2),(\Phi_3,\Psi_3)\},
Sh(S_\Sigma)
\Bigr),
\qquad
Sh(S_\Sigma)\cong \mathcal P.
\]

The extended vertex set is
\[
V_\Sigma^{\mathrm{ext}}:=\{v_1,v_2,v_3,v_{\mathrm{bulk}}\}.
\]
The direct bulk/localized coupling relation contributes the six ordered pairs
\[
(v_i,v_{\mathrm{bulk}}),
\qquad
(v_{\mathrm{bulk}},v_i),
\qquad
1\le i\le 3.
\]
The mediated relation is determined by the nine composites
\[
\Psi_j\circ\Phi_i:\mathcal C_{p_i}\to \mathcal C_{p_j},
\qquad
1\le i,j\le 3.
\]
Accordingly, after ordering the extended vertex set as
\[
\{v_1,v_2,v_3,v_{\mathrm{bulk}}\},
\]
the binary incidence matrix takes the form
\begin{equation}\label{eq:three-node-incidence-matrix-general}
I_\Sigma=
\begin{pmatrix}
1 & 1 & 1 & 1\\
1 & 1 & 1 & 1\\
1 & 1 & 1 & 1\\
1 & 1 & 1 & 0
\end{pmatrix}.
\end{equation}
The upper-left $3\times 3$ block is the binary decategorified image of the mediated node-to-node coupling structure, while the last row and last column encode the framing by the bulk sector.

Thus the decategorified incidence package is
\[
\mathcal I_\Sigma=
\left(
\{v_1,v_2,v_3,v_{\mathrm{bulk}}\},
\begin{pmatrix}
1 & 1 & 1 & 1\\
1 & 1 & 1 & 1\\
1 & 1 & 1 & 1\\
1 & 1 & 1 & 0
\end{pmatrix}
\right),
\]
and the quiver-theoretic package is
\begin{equation}\label{eq:three-node-quiver-package}
\mathfrak Q_\Sigma=
\left(
\{v_1,v_2,v_3\},
\Q e_1\oplus\Q e_2\oplus\Q e_3,
(c_1,c_2,c_3),
\{(\Phi_1,\Psi_1),(\Phi_2,\Psi_2),(\Phi_3,\Psi_3)\},
\begin{pmatrix}
1 & 1 & 1 & 1\\
1 & 1 & 1 & 1\\
1 & 1 & 1 & 1\\
1 & 1 & 1 & 0
\end{pmatrix}
\right).
\end{equation}

\begin{proposition}\label{prop:three-node-incidence-example}
Let $\pi:X\to\Delta$ be a finite-node conifold degeneration with node set $\Sigma=\{p_1,p_2,p_3\}$. Then the associated functorial incidence package has extended vertex set
\[
V_\Sigma^{\mathrm{ext}}=\{v_1,v_2,v_3,v_{\mathrm{bulk}}\}
\]
and binary incidence matrix of the form \eqref{eq:three-node-incidence-matrix-general}. In particular:
\begin{enumerate}
\item the last row and last column encode the direct bulk/localized coupling pattern;
\item the upper-left $3\times 3$ block records the full mediated binary coupling pattern among the three localized sectors.
\end{enumerate}
\end{proposition}

\begin{proof}
The direct bulk/localized entries are forced by the three pairs of attachment functors $(\Phi_i,\Psi_i)$, while the localized-localized entries are exactly the characteristic values of the nine possible mediated composites $\Psi_j\circ\Phi_i$. Since the extended vertex set has four elements, the incidence matrix is $4\times 4$ and has the claimed block form.
\end{proof}

\begin{remark}
\label{rem:three-node-pattern}
The three-node case makes the general pattern completely transparent. The binary incidence matrix naturally decomposes into a localized-localized mediated block and a bulk framing row and column. This is the finite interaction pattern that later papers will refine by introducing stability and BPS data.
\end{remark}

\subsection{General finite-node pattern}

The preceding examples display the general shape of the current work output. Let
\[
\Sigma=\{p_1,\dots,p_r\}
\]
be the finite node set. Then:
\begin{enumerate}
\item the state-data package is
\[
A_\Sigma=
\bigl(
\{v_1,\dots,v_r\},
\bigoplus_{k=1}^r \Q e_k,
(c_1,\dots,c_r)
\bigr);
\]
\item the extended vertex set is
\[
V_\Sigma^{\mathrm{ext}}:=\{v_1,\dots,v_r,v_{\mathrm{bulk}}\};
\]
\item the direct bulk/localized coupling relation contributes the ordered pairs
\[
(v_i,v_{\mathrm{bulk}}),
\qquad
(v_{\mathrm{bulk}},v_i),
\qquad
1\le i\le r;
\]
\item the mediated relation is determined by the composites
\[
\Psi_j\circ\Phi_i:\mathcal C_{p_i}\to \mathcal C_{p_j},
\qquad
1\le i,j\le r;
\]
\item after choosing an ordering of the extended vertex set, the incidence matrix takes the framed block form
\begin{equation}\label{eq:general-framed-incidence-matrix}
I_\Sigma=
\begin{pmatrix}
\chi_{11} & \cdots & \chi_{1r} & 1\\
\vdots & \ddots & \vdots & \vdots\\
\chi_{r1} & \cdots & \chi_{rr} & 1\\
1 & \cdots & 1 & 0
\end{pmatrix},
\end{equation}
where each entry \(\chi_{ij}\in\{0,1\}\) records the presence or absence of the mediated coupling \(v_i\rightsquigarrow_{\mathrm{med}}v_j\). Under the present mediated-relation convention, one has \(\chi_{ij}=1\) for all \(1\le i,j\le r\).
\end{enumerate}

Thus the general quiver-theoretic package has the form
\[
\mathfrak Q_\Sigma=
\bigl(
V_\Sigma,
E_\Sigma,
c_\Sigma,
\mathcal F_\Sigma,
I_\Sigma
\bigr),
\]
with \(I_\Sigma\) of framed binary block type as in \eqref{eq:general-framed-incidence-matrix}.

\begin{proposition}\label{prop:general-framed-pattern}
For every finite-node conifold degeneration with \(|\Sigma|=r\), the binary incidence matrix \(I_\Sigma\) is of framed block type: an \(r\times r\) mediated block, together with a distinguished bulk row and bulk column encoding the direct bulk/localized couplings.
\end{proposition}

\begin{proof}
This follows immediately from the construction of the direct bulk/localized coupling relation and the mediated coupling relation, together with the definition of the binary incidence matrix as the characteristic matrix of their union.
\end{proof}

\begin{remark}
\label{rem:examples-role-partII}
The examples above do not yet assign weighted arrow multiplicities. Their role is to make completely explicit the finite binary interaction layer extracted in the present work. This is precisely the level of structure proved in the present paper and exactly the input needed for the later stability, BPS-spectrum, and wall-crossing developments.
\end{remark}

\section{Conclusion}

The purpose of the present paper was to extract, from the finite-node schober realization of the degeneration, the interaction and incidence layer that lies between the algebraic state-data package of \cite{RahmanQuiverDataI} and the later stability/BPS formalism. Starting from the finite-node input data
\[
A_\Sigma:=\bigl(V_\Sigma,E_\Sigma,c_\Sigma\bigr),
\qquad
S_\Sigma:=\Bigl(\mathcal C_{\mathrm{bulk}},\{\mathcal C_{p_k}\}_{k=1}^r,\{\Phi_k,\Psi_k\}_{k=1}^r,Sh(S_\Sigma)\Bigr),
\]
we constructed the functorial coupling relation, the functorial incidence package
\[
\mathfrak{I}_\Sigma:=\bigl(V_\Sigma^{\mathrm{ext}},\rightsquigarrow_\Sigma\bigr),
\]
the decategorified incidence package
\[
\mathcal I_\Sigma:=\bigl(V_\Sigma^{\mathrm{ext}},I_\Sigma\bigr),
\]
and the resulting finite quiver-theoretic package
\[
\mathfrak Q_\Sigma:=\bigl(V_\Sigma,E_\Sigma,c_\Sigma,\mathcal F_\Sigma,I_\Sigma\bigr),
\qquad
\mathcal F_\Sigma:=\{(\Phi_k,\Psi_k)\}_{k=1}^r.
\]
We proved that this package is canonically determined by the finite-node schober datum, compatible with the corrected perverse extension and its mixed-Hodge-module refinement, and intrinsic under the appropriate notion of equivalence of finite-node schober realizations.

Thus the present paper provides the interaction and incidence layer required for the next stage of the sequence. Later work will refine the interaction layer of \(\mathfrak Q_\Sigma\) and equip it with stability data, BPS sectors, BPS indices, and wall-crossing structure.

%
%
\appendix

\section{On binary decategorification and weighted refinements}
\label{app:binary-decategorification}

The purpose of this appendix is to justify the binary decategorification adopted in Section~\ref{sec:decategorified-incidence-data}, to state precisely what information it retains and what information it discards, and to formulate the mathematical conditions under which a later weighted refinement could replace it. The point is not merely expository. The passage
\[
\mathfrak{I}_\Sigma:=\bigl(V_\Sigma^{\mathrm{ext}},\rightsquigarrow_\Sigma\bigr)
\quad\leadsto\quad
I_\Sigma\in M_{r+1}(\{0,1\})
\]
is a mathematically consequential choice: it yields a canonical and canonical incidence object, but it deliberately stops short of assigning weighted interaction numbers. The present appendix explains why this is the correct choice at the level of the theorem package established in the body of the paper.

\subsection{The binary decategorification}

Recall from Section~\ref{sec:functorial-incidence-package} that the finite-node schober package determines the functorial incidence package
\[
\mathfrak{I}_\Sigma:=\bigl(V_\Sigma^{\mathrm{ext}},\rightsquigarrow_\Sigma\bigr),
\qquad
V_\Sigma^{\mathrm{ext}}:=V_\Sigma\sqcup\{v_{\mathrm{bulk}}\}.
\]
The decategorification adopted in Section~\ref{sec:decategorified-incidence-data} is the characteristic function of the relation \(\rightsquigarrow_\Sigma\).

\begin{definition}[binary decategorification]
\label{def:binary-decategorification-appendix}
The \emph{binary decategorification} of the functorial incidence package
\[
\mathfrak{I}_\Sigma:=\bigl(V_\Sigma^{\mathrm{ext}},\rightsquigarrow_\Sigma\bigr)
\]
is the characteristic function
\[
\chi_\Sigma:V_\Sigma^{\mathrm{ext}}\times V_\Sigma^{\mathrm{ext}}\to \{0,1\}
\]
defined by
\begin{equation}\label{eq:appendix-binary-characteristic}
\chi_\Sigma(u,v):=
\begin{cases}
1,&\text{if }u\rightsquigarrow_\Sigma v,\\[2mm]
0,&\text{if }u\not\rightsquigarrow_\Sigma v.
\end{cases}
\end{equation}
After choosing an ordering of \(V_\Sigma^{\mathrm{ext}}\), the associated matrix
\[
I_\Sigma:=\bigl(\chi_\Sigma(u_\alpha,u_\beta)\bigr)
\]
is the \emph{binary incidence matrix}.
\end{definition}

\begin{remark}
\label{rem:binary-means-support-only}
Definition~\ref{def:binary-decategorification-appendix} should be read literally. The binary incidence matrix records only the support of the coupling pattern:
\[
\text{present}
\qquad\text{or}\qquad
\text{absent}.
\]
It does not yet record weighted interaction numbers, numerical arrow multiplicities, Euler characteristics, ranks on \(K_0\), or skew pairings.
\end{remark}

\begin{proposition}[canonicality of the binary decategorification]
\label{prop:canonicality-binary-decategorification}
The binary decategorification of Definition~\ref{def:binary-decategorification-appendix} is canonically determined by the finite-node schober package. In particular, it involves no auxiliary choices beyond the ordering convention used to display the matrix \(I_\Sigma\).
\end{proposition}

\begin{proof}
By Theorem~\ref{thm:canonical-functorial-incidence-package}, the finite-node schober package canonically determines the relation
\[
\rightsquigarrow_\Sigma\subseteq V_\Sigma^{\mathrm{ext}}\times V_\Sigma^{\mathrm{ext}}.
\]
The characteristic function of a specified relation is uniquely determined. Hence the function \(\chi_\Sigma\) is canonical. The only non-canonical step is the choice of ordering of the finite set \(V_\Sigma^{\mathrm{ext}}\) when one writes \(\chi_\Sigma\) in matrix form. Different orderings change \(I_\Sigma\) only by simultaneous permutation of rows and columns.
\end{proof}

\begin{proposition}[invariance under equivalence]
\label{prop:appendix-binary-invariance}
Let \(S_\Sigma\) and \(T_\Sigma\) be equivalent finite-node schober realizations in the sense of Section~\ref{sec:invariance-under-equivalence}. Then the corresponding binary incidence matrices agree up to simultaneous permutation of rows and columns. Equivalently, the binary decategorification is intrinsic to the finite-node schober equivalence class.
\end{proposition}

\begin{proof}
By Theorem~\ref{thm:invariance-functorial-incidence-package}, equivalent finite-node schober realizations determine the same functorial incidence package
\[
\mathfrak{I}_\Sigma:=\bigl(V_\Sigma^{\mathrm{ext}},\rightsquigarrow_\Sigma\bigr)
\]
up to the canonical identification induced by the common node set. The characteristic function of a relation is invariant under such an identification, and passing to matrix form changes only by simultaneous permutation of rows and columns.
\end{proof}

\subsection{What the binary choice retains}

The binary decategorification retains exactly the following data.

\begin{proposition}[support-level information retained]
\label{prop:binary-retains-support}
The binary incidence matrix \(I_\Sigma\) retains:
\begin{enumerate}
\item the direct bulk/localized coupling pattern determined by the attachment functors
\[
\Phi_k:\mathcal C_{p_k}\to \mathcal C_{\mathrm{bulk}},
\qquad
\Psi_k:\mathcal C_{\mathrm{bulk}}\to \mathcal C_{p_k};
\]
\item the mediated localized-to-localized coupling pattern determined by the composites
\[
\Psi_j\circ\Phi_i:\mathcal C_{p_i}\to \mathcal C_{p_j};
\]
\item the finite support of the interaction law on the extended vertex set
\[
V_\Sigma^{\mathrm{ext}}:=V_\Sigma\sqcup\{v_{\mathrm{bulk}}\}.
\]
\end{enumerate}
\end{proposition}

\begin{proof}
Each matrix entry of \(I_\Sigma\) is by definition the characteristic value of the corresponding coupling relation. Therefore \(I_\Sigma\) records exactly which ordered pairs in
\[
V_\Sigma^{\mathrm{ext}}\times V_\Sigma^{\mathrm{ext}}
\]
belong to the direct bulk/localized coupling relation or to the mediated relation. This is precisely the support of the interaction pattern.
\end{proof}

\begin{remark}
\label{rem:binary-is-topology-of-coupling}
The mathematical content of Proposition~\ref{prop:binary-retains-support} may be summarized informally as follows:
\[
\text{binary incidence}
=
\text{support of the coupling pattern}.
\]
That is, the binary package is the topology of the interaction law, not yet its full numerical strength.
\end{remark}

\subsection{What the binary choice discards}

The cost of binary decategorification is the loss of quantitative information.

\begin{proposition}[information discarded by binary decategorification]
\label{prop:binary-discards-weights}
The binary incidence matrix \(I_\Sigma\) does not, by itself, determine:
\begin{enumerate}
\item the number of independent interaction channels between two vertices;
\item the rank of an induced map on a Grothendieck group \(K_0\);
\item the dimension of a linearized image attached to a composite functor;
\item the Euler characteristic of a Hom-complex;
\item a skew-symmetric or signed interaction pairing;
\item numerical arrow multiplicities in a weighted quiver.
\end{enumerate}
\end{proposition}

\begin{proof}
By construction, each entry of \(I_\Sigma\) takes only the values \(0\) or \(1\). Therefore two categorically distinct situations with the same support pattern yield the same binary entry. In particular, the matrix does not distinguish between one interaction channel and several, between weak and strong interactions, or between different possible numerical decategorifications of the same relation.
\end{proof}

\begin{remark}
\label{rem:binary-is-coarse}
Proposition~\ref{prop:binary-discards-weights} explains why the binary package is structurally correct but quantitatively coarse. It is the minimal rigorous interaction object available from the theorem package proved in the paper, but it is not yet the full numerical interaction law one would ultimately want for stability, BPS multiplicities, or wall-crossing coefficients.
\end{remark}

\subsection{Weighted refinements: abstract framework}

We now formulate the abstract notion of a weighted refinement of the binary incidence matrix.

\begin{definition}[weighted refinement]
\label{def:weighted-refinement}
A \emph{weighted refinement} of the binary incidence matrix \(I_\Sigma\) consists of a matrix
\[
I_\Sigma^{\mathrm{wt}}:=\bigl(w_{\alpha\beta}\bigr)
\]
with entries in \(\mathbf Z_{\ge 0}\), \(\mathbf Z\), or \(\mathbf Q\), together with a rule assigning to each ordered pair of vertices a weight
\[
w_{\alpha\beta}:=\mathcal J_{\alpha\beta},
\]
such that:
\begin{enumerate}
\item \(w_{\alpha\beta}\) is defined from the finite-node schober interaction data by a canonically specified invariant;
\item \(I_\Sigma^{\mathrm{wt}}\) is invariant under equivalence of finite-node schober realizations;
\item the support of \(I_\Sigma^{\mathrm{wt}}\) agrees with the binary incidence matrix in the sense that
\begin{equation}\label{eq:weighted-support-condition}
w_{\alpha\beta}=0
\quad\Longleftrightarrow\quad
I_\Sigma(\alpha,\beta)=0.
\end{equation}
\end{enumerate}
\end{definition}

\begin{remark}
\label{rem:weighted-refinement-purpose}
Condition \eqref{eq:weighted-support-condition} says that the binary package should be the support of any later weighted refinement. Thus the philosophy is
\[
I_\Sigma^{(0/1)}\subset I_\Sigma^{\mathrm{wt}},
\]
where the first matrix records where interactions are allowed and the second records how strong or how numerous those interactions are.
\end{remark}

\begin{proposition}[support of a weighted refinement]
\label{prop:support-of-weighted-refinement}
Let \(I_\Sigma^{\mathrm{wt}}\) be a weighted refinement in the sense of Definition~\ref{def:weighted-refinement}. Then the binary incidence matrix is recovered from \(I_\Sigma^{\mathrm{wt}}\) by the support rule
\begin{equation}\label{eq:recover-binary-from-weighted}
I_\Sigma(\alpha,\beta)=
\begin{cases}
1,&\text{if }w_{\alpha\beta}\neq 0,\\[2mm]
0,&\text{if }w_{\alpha\beta}=0.
\end{cases}
\end{equation}
\end{proposition}

\begin{proof}
This is just a restatement of the support condition \eqref{eq:weighted-support-condition}.
\end{proof}

\subsection{When a weighted refinement becomes available}

The binary decategorification can be replaced only after a specific invariant has been chosen and justified.

\begin{proposition}[criterion for replacing the binary package]
\label{prop:criterion-for-weighted-replacement}
Suppose there exists a canonically chosen invariant \(\mathcal J\) assigning to each relevant direct or mediated coupling a weight
\[
w_{ij}:=\mathcal J(\Psi_j\circ\Phi_i)
\]
or, in the bulk/localized case, a corresponding weight attached to \(\Phi_i\) or \(\Psi_i\), such that:
\begin{enumerate}
\item \(w_{ij}\) is well-defined;
\item \(w_{ij}\) is invariant under finite-node schober equivalence;
\item \(w_{ij}\) is compatible with the corrected perverse shadow and the state-data package of \cite{RahmanQuiverDataI};
\item the support condition \eqref{eq:weighted-support-condition} holds.
\end{enumerate}
Then the binary incidence matrix \(I_\Sigma\) may be replaced by the weighted incidence matrix
\[
I_\Sigma^{\mathrm{wt}}:=\bigl(w_{\alpha\beta}\bigr),
\]
and the resulting weighted package is a refinement of the binary one.
\end{proposition}

\begin{proof}
Under hypotheses (1) and (2), the matrix \(I_\Sigma^{\mathrm{wt}}\) is a well-defined and intrinsic algebraic object. Hypothesis (3) ensures that it remains tied to the finite-node architecture already established in \cite{RahmanQuiverDataI} and the present work. Hypothesis (4) implies that the binary matrix \(I_\Sigma\) is exactly the support of \(I_\Sigma^{\mathrm{wt}}\). Therefore the weighted matrix is a legitimate refinement of the binary package.
\end{proof}

\begin{remark}
\label{rem:moment-of-refinement}
Proposition~\ref{prop:criterion-for-weighted-replacement} identifies the exact mathematical moment at which the binary package can be replaced: it occurs when one has chosen and justified a specific decategorification invariant \(\mathcal J\) satisfying the four conditions above. Before that moment, the binary package is the correct stopping point.
\end{remark}

\subsection{Candidate sources of weighted invariants}

The present paper does not choose a weighted invariant. Nevertheless, it is useful to record the most natural candidates.

\begin{remark}
\label{rem:possible-weighted-invariants}
Among the natural candidates for a later invariant \(\mathcal J\) are:
\begin{enumerate}
\item \emph{Grothendieck-group decategorification:} if the categories involved admit sufficiently controlled Grothendieck groups, one may try to define \(w_{ij}\) from the induced maps
\[
K_0(\mathcal C_{p_i})\to K_0(\mathcal C_{\mathrm{bulk}}),
\qquad
K_0(\mathcal C_{\mathrm{bulk}})\to K_0(\mathcal C_{p_j}),
\]
or from the rank of the resulting composite;
\item \emph{Hom/Euler-type decategorification:} one may try to define
\[
w_{ij}:=\dim \mathcal H_{ij}
\qquad\text{or}\qquad
w_{ij}:=\chi(\mathcal H_{ij})
\]
for a canonically attached Hom-space or Hom-complex associated with the composite \(\Psi_j\circ\Phi_i\);
\item \emph{Geometric/Hodge-theoretic bridge invariants:} one may try to tie the composite \(\Psi_j\circ\Phi_i\) back to the nodewise extension channels, the coefficient vector \(c_\Sigma\), or a Hodge-theoretic quantity attached to the corresponding local-to-global coupling.
\end{enumerate}
Any such choice would require a separate theorem package proving well-definedness, equivalence invariance, and compatibility with the corrected perverse shadow.
\end{remark}

\begin{remark}
\label{rem:binary-not-final}
The binary package should therefore be viewed as the first rigorous floor rather than the final interaction law. Its role is to provide the support of the future weighted interaction theory. Later work may refine
\[
I_\Sigma\in M_{r+1}(\{0,1\})
\]
to a richer matrix
\[
I_\Sigma^{\mathrm{wt}},
\]
but only after the relevant invariant has been mathematically justified.
\end{remark}

\subsection{Binary choice}

We summarize the logic of the binary choice.

\begin{theorem}[status of the binary choice]
\label{thm:status-of-binary-choice}
At the level of the theorem package established in the present paper, the binary incidence matrix
\[
I_\Sigma\in M_{r+1}(\{0,1\})
\]
is the unique canonical decategorified interaction object forced by the finite-node schober package. It is canonical, intrinsic under schober equivalence, and sufficient to support the quiver-theoretic assembly of the present work. Any weighted refinement requires additional structure beyond what has presently been proved.
\end{theorem}

\begin{proof}
The canonicality and intrinsicity statements follow from Propositions~\ref{prop:canonicality-binary-decategorification} and \ref{prop:appendix-binary-invariance}. The fact that \(I_\Sigma\) supports the quiver-theoretic assembly is exactly the content of Sections~\ref{sec:decategorified-incidence-data} and \ref{sec:quiver-theoretic-package}. The final statement follows from Proposition~\ref{prop:criterion-for-weighted-replacement}.
\end{proof}

\begin{remark}
\label{rem:appendix-final-interpretation}
The binary choice is therefore mathematically conservative but conceptually precise. It keeps the interaction layer fully rigorous at the present stage, while also making explicit that the eventual BPS/quiver dynamics will require a later weighted refinement. In this sense, the binary incidence matrix is the support of the future interaction law, not its final form.
\end{remark}

%
%
\printbibliography
\end{document}